\documentclass[11pt,twoside]{amsart}

\usepackage{amsmath,amssymb,amsthm,mathtools}
\usepackage{enumerate} 
\usepackage{verbatim}   
\usepackage{stmaryrd}
\usepackage{xcolor}
\usepackage{graphicx}
\usepackage{tikz}
\usetikzlibrary{decorations.pathreplacing}
\usepackage{tikz-cd}

\usepackage[lite]{amsrefs}

\usepackage[colorlinks=true,linkcolor=blue,citecolor=blue,urlcolor=blue]{hyperref}

\BibSpec{collection.article}{%
  +{}  {\PrintAuthors}                {author}
  +{,} { \textit}                     {title}
  +{.} { }                            {part}
  +{:} { \textit}                     {subtitle}
  +{,} { \PrintContributions}         {contribution}
  +{,} { \PrintConference}            {conference}
  +{}  {\PrintBook}                   {book}
  +{,} { }                            {booktitle}
  +{,} { }                            {series}
  +{, vol.} { }                       {volume}
  +{,} { }                            {publisher}
  +{,} { \PrintDateB}                 {date}
  +{,} { pp.~}                        {pages}
  +{,} { }                            {status}
  +{,} { \PrintDOI}                   {doi}
  +{,} { available at \eprint}        {eprint}
  +{}  { \parenthesize}               {language}
  +{}  { \PrintTranslation}           {translation}
  +{;} { \PrintReprint}               {reprint}
  +{.} { }                            {note}
  +{.} {}                             {transition}
  +{}  {\SentenceSpace \PrintReviews} {review}
}

\AtBeginDocument{\def\MR#1{}}

\newcommand{\lto}{\longrightarrow}

\DeclareMathOperator{\lcm}{lcm}

\DeclareMathOperator{\Spec}{Spec}

\DeclareMathOperator{\Proj}{Proj}
\DeclareMathOperator{\Quot}{Quot}
\DeclareMathOperator{\rank}{rank}

\newcommand{\cF}{\mathcal{F}}

\newcommand{\cM}{\mathcal{M}}
\newcommand{\cN}{\mathcal{N}}
\newcommand{\cO}{\mathcal{O}}
\newcommand{\cP}{\mathcal{P}}

\newcommand{\cR}{\mathcal{R}}
\newcommand{\cS}{\mathcal{S}}

\newcommand{\fP}{\mathfrak{P}}
\newcommand{\fQ}{\mathfrak{Q}}

\newcommand{\m}{\mathfrak{m}}
\newcommand{\p}{\mathfrak{p}}
\newcommand{\q}{\mathfrak{q}}

\newcommand{\bC}{\mathbb{C}}

\newcommand{\bN}{\mathbb{N}}
\newcommand{\bP}{\mathbb{P}}
\newcommand{\bZ}{\mathbb{Z}}
\newcommand{\bR}{\mathbb{R}}
\newtheorem{Theorem}{Theorem}[section]
\newtheorem{Lemma}[Theorem]{Lemma}
\newtheorem{Corollary}[Theorem]{Corollary}
\newtheorem{Proposition}[Theorem]{Proposition}
\newtheorem{Remark}[Theorem]{Remark}

\newtheorem{Example}[Theorem]{Example}

\newtheorem{Notation}[Theorem]{Notation}
\newtheorem{Discussion}[Theorem]{Discussion}
\newtheorem{NotationDiscussion}[Theorem]{Notation and Discussion}

\makeatletter
\let\c@equation\c@Theorem
\makeatother

\begin{document}

\baselineskip=16pt

 \title[Behrend function and blowup algebras]{\bf 
Behrend function and blowup algebras}
\date\today

\author[ Claudia Polini, Alessio Sammartano, and Bernd Ulrich]
{ Claudia Polini, Alessio Sammartano, and Bernd Ulrich}

\thanks{MSC 2020 {\em Mathematics Subject Classification}.
Primary 13A30
Secondary 05E40, 11D04, 13A18, 13B22, 13F55, 13F65, 13H15, 14B05, 14M25, 20M14, 52B20.
}

\thanks{
The first and third authors were partially supported by NSF grants DMS-2502707 and DMS-2502706, respectively.
The second author  was  supported by the grants
PRIN2020355B8Y \emph{Squarefree Gr\"obner degenerations, special varieties and related topics} and 
PRIN 2022K48YYP \emph{Unirationality, Hilbert schemes, and singularities},
and is a member of  INdAM – GNSAGA}


\address{Department of Mathematics, 
University of Notre Dame,
Notre Dame, IN 46556, USA} \email{cpolini@nd.edu}

\address{Dipartimento di Matematica,
Politecnico di Milano,
Milano, 20133, Italy}
\email{alessio.sammartano@polimi.it}

\address{Department of Mathematics,
Purdue University,
West Lafayette, IN 47907, USA}\email{bulrich@purdue.edu}

\begin{abstract}
Given a scheme $X$ of finite type over the complex numbers, the {\it Behrend function} is a constructible function 
$\nu_X: X(\mathbb C) \rightarrow \mathbb Z \, $
introduced by Behrend in order to   define enumerative invariants in Donaldson--Thomas theory. 
Even in simple cases, the Behrend function is very difficult to compute.
In this article, we tackle the problem of computing the Behrend function of zero-dimensional schemes.
We obtain a number of explicit formulas, in particular, for arbitrary zero-dimensional monomial schemes, 
thus providing vast generalizations of previous work of Graffeo--Ricolfi.
Our main tools come from the theory of blowup algebras. 
Along the way, we establish results of independent interest related to the integer decomposition property, weighted Veronese subrings, and reduced fiber rings.
\end{abstract}
 
\maketitle

\section{Introduction}\label{SectionIntro}

The Behrend function, introduced  in~\cite{Beh}, is a 
 constructible function $\nu_X \colon X(\bC) \to \bZ$ canonically associated 
to any scheme $X$ of finite type over the complex numbers. 
Serving as a refined measure of the singularities of $X$,
a key feature of $\nu_X$ is  the identity
\[
\deg [X]^{\mathrm{vir}} \;=\; \sum_{m \in \bZ} m \cdot \chi\!\left(\nu_X^{-1}(m)\right),
\]
which expresses the degree of the virtual fundamental class of $X$ as a 
weighted Euler characteristic
when $X$ is proper and admits a symmetric perfect obstruction theory.
Unlike its standard counterpart, 
this weighted Euler characteristic is highly sensitive to the local schematic structure of $X$.
Moreover,
while the left-hand side of this identity requires $X$ to be proper to be meaningful, 
the right-hand side is well-defined without any properness assumption, 
allowing a definition of Donaldson--Thomas invariants on non proper schemes. 
Consequently, the introduction of the Behrend function underlies a wide range of developments in 
modern enumerative geometry; see~\cites{Pandharipande, Szendroi} for surveys.

In spite of its importance, the Behrend function has proven 
remarkably difficult to compute, even in seemingly simple situations. 
A special case is
when $X$ is locally the critical locus of a regular function $f$ on a smooth scheme;
for such schemes, $\nu_X$
 can be expressed in terms of the Milnor fiber of $f$.
A notable example of this setup is the Hilbert scheme of points of a smooth Calabi-Yau threefold
\cites{BBS,Ricolfi},
for which the conjectured constancy of $\nu_X$ was recently disproved
\cite{JKS}.
Another   class of schemes for which $\nu_X$ can be computed is zero-dimensional complete intersections.
Beyond these specific classes, however, the general problem has  remained largely out of reach.
A breakthrough came with the work of 
Graffeo and Ricolfi~\cite{Graffeo}, who carried out the first 
systematic study of the Behrend function for zero-dimensional schemes 
defined by ideals in two variables, obtaining explicit formulas in 
several 
cases: 
monomial ideals that are the integral closure of ideals generated by pure powers of the variables,
a class of ideals they call \emph{towers}, and 
products of towers. Their methods rely on a fine analysis of blowups, 
normalizations, and the Dynkin diagrams of exceptional curves, drawing 
heavily on toric geometry techniques specific to surfaces. As a 
consequence, the bulk of their results is restricted to ideals in 
$\bC[x,y]$; in higher dimensions, as they themselves observe~\cite{Graffeo}*{Section~7}, 
their toric methods are not directly applicable, and a finer analysis 
is required.

In this paper, we revisit and substantially extend their results, 
guided by a different perspective: \emph{the Behrend function of a 
zero-dimensional scheme is governed by the blowup algebras of its 
defining ideal.} This point of view is the conceptual contribution of the paper. Rather than relying on techniques specific to surfaces, we show that the fundamental invariants governing the Behrend function arise naturally from the blowup algebra, allowing the theory to extend uniformly to arbitrary embedding dimension.
More precisely, let $I \subset R = \bC[x_1,\ldots,x_n]$ be a zero-dimensional ideal.
Without loss of generality, we may assume that $I$ is primary to $(x_1, \ldots, x_n)$.
The Behrend function of $X=\mathrm{Spec}(R/I)$ is determined by the Rees ring 
${\mathcal R}(I) = R[It]$, its integral closure $\overline{\mathcal R(I)}$, 
and the associated graded ring $G = \mathrm{gr}_I(R).$
This shift of perspective 
allows us to employ the full machinery of commutative 
algebra 
and yields formulas of considerably greater generality, 
valid in any embedding dimension, with proofs that are both shorter 
and more conceptually transparent.

\subsection*{Main results}
Let $I \subset R$ be an ideal primary to $(x_1, \ldots, x_n)$.
For simplicity, we write $\nu_I$ for the value of the  Behrend function of  $X=\mathrm{Spec}(R/I)$ at the unique closed point.
Our starting point is Behrend's formula \eqref{EqDefBehrendFunction} expressing $\nu_I$ in terms of 
the minimal primes of $G$. For any ideal $I \subset R$ we prove 
(Proposition~\ref{PropLength}) the length formula
\[
\lambda(G_{\fP}) \;=\; \sum_j v_j(I) \cdot 
\bigl[(\overline{\mathcal{R}}/{\fQ}_j) : (\mathcal{R}/{\fP})\bigr],
\]
where $\fP$ is a minimal prime of $I\mathcal{R}$, the $\fQ_j$ are the primes 
of $\overline{\mathcal{R}}$ lying over $\fP$, and the $v_j$ are the 
corresponding Rees valuations of $I$. 
This  isolates the degrees of the residue field 
extensions as the central quantities to be computed.

The technical heart of the paper is the explicit determination of these 
field degrees for monomial ideals in any number of variables. When $I$ 
is monomial, the Rees valuations of $I$ correspond to the non-coordinate 
supporting hyperplanes of the Newton polyhedron ${\rm NP}(I)$, and the centers 
of these valuations are generated by subsets of the variables that can be easily obtained 
from the hyperplanes. Exploiting this combinatorial structure together 
with Ehrhart theory and a careful analysis of the relationship between 
Veronese subrings of $R$ and the Ehrhart ring of the relevant facets of 
${\rm NP}(I)$, we obtain (Theorem~\ref{TheoremBehrensMonomials}) the formulas
\[
\nu_I \;=\; \sum_{i} v_i(I) \cdot |\mathcal{I}_n(\log \mathcal{A}_i')|
\;=\; \sum_{i} v_i(I) \cdot |\mathcal{I}_{n-1}(\log \mathcal{A}_i'')|,
\]
where the sum runs over all Rees valuations $v_i$ of $I$, 
$\log \mathcal{A}_i'$, $\log \mathcal{A}_i''$ are explicit $\mathbb{Z}$-valued \emph{log matrices} 
attached to the corresponding compact facets of ${\rm NP}(I)$,
and $|\mathcal{I}_p(-)|$ denotes the $\gcd$ of the minors of size $p$. 
Thus, the entire 
Behrend function is encoded by the Rees valuations together with a 
 matrix invariant for each valuation. This formula applies to 
\emph{every} zero-dimensional monomial ideal in any number of 
variables, vastly generalizing the formulas of~\cite{Graffeo}, which 
apply to ideals in $\bC[x,y]$.

\subsection*{The normal case: short proofs}

An immediate consequence of our main formula 
(Corollary~\ref{CorFormulaBehrendWithFieldDegreesNormalCase}) is that 
for an arbitrary normal zero-dimensional ideal one has the  
simple expression
\[
\nu_I \;=\; \sum_v v(I),
\]
where the sum ranges over the Rees valuations of $I$. Combined with classical 
criteria for normality, this immediately gives closed formulas for several 
classes of ideals studied in~\cite{Graffeo}, with proofs that are 
considerably shorter. To illustrate:

\begin{itemize}
\item \emph{Integrally closed monomial ideals in two variables} 
(Examples~\ref{ExampleBehrendNormalTwoVariablesI} 
and~\ref{ExampleBehrendNormalTwoVariablesII}). The full formula follows from a quick calculation involving the slopes of the boundary of ${\rm NP}(I)$. 
The case of a single Rees valuation 
recovers~\cite{Graffeo}*{Theorem~B}, and the general case 
recovers~\cite{Graffeo}*{Theorem~5.13}.

\item \emph{Towers} (Example~\ref{ExBehrendTowers}). Our formula 
recovers~\cite{Graffeo}*{Theorems~3.11 and 3.13},
replacing the original combinatorial argument that occupies 
\cite{Graffeo}*{Section~3}.

\item \emph{Products of complete towers} (Examples \ref{ExProdTowersI} and \ref{ExProdTowersII}). Our results provide short proofs of the formulas of
\cite{Graffeo}*{Theorem~3.17}. 
\end{itemize}

In each case, our argument simply identifies the Rees valuations of 
the ideal---in some examples requiring a coordinate change, in others 
directly visible from the Newton polyhedron---and adds up the values 
$v(I)$. 

Even more strikingly,
we show that the same simple formula 
$\nu_I = \sum v(I)$ is valid for \emph{integrally closed} monomial ideals 
in three variables (Corollary~\ref{Corn3}), even though such ideals are not  normal in general.

\subsection*{An unexpectedly simple formula for ideals with one Rees 
valuation}

In Section~\ref{SectionOneVal}, we show
that zero-dimensional monomial ideals possessing only one Rees valuation behave 
remarkably well, essentially as if they were normal, even when they 
are far from being normal. 
This is unexpected for the following reason. 
In the fourth formula of Theorem~\ref{TheoremBehrensMonomials},
each Rees valuation $v_i$ of $I$ contributes
 a factor $|\mathcal{I}_n(\log \mathcal{A}_i)|/|\mathcal{I}_n(\log \mathcal{C}_i)|$ 
as well as   an auxiliary positive integer $s_i$,
which we call the \emph{birational Veronese index} of $I$ relative to  $v_i$.
The whole contribution reduces to the much simpler quantity
$|\mathcal{I}_n(\log \mathcal{A}_i)|/|\mathcal{I}_n(\log \mathcal{B}_i)|$ 
precisely when a certain Veronese-type extension is birational, or equivalently, when $s_i = 1$. 
In general, 
this birationality fails, even for integrally closed zero-dimensional monomial ideals (see 
Example~\ref{E5.10}). Surprisingly, however, we prove that it always 
holds for zero-dimensional monomial ideals having a single Rees valuation, i.e. for monomial ideals integral over an ideal generated by powers of the variables (Corollary ~\ref{CorBiratVerIndexOneVal}). 

The proof rests on a result in additive combinatorics 
(Theorem~\ref{ThmSameSubgroups}): given positive integers $\ell_1, \ldots, \ell_n$ and a positive common multiple $M$, the set of nonnegative integer 
solutions of the equation $\, \sum \ell_j u_j = M$ generates the same subgroup of $\bZ^n$ 
as the set of  solutions of the linear congruence $\, \sum \ell_j u_j \equiv 0 \  {\rm mod} \, M$. 
This result is remarkable, 
as it is a weaker version of the \emph{integer 
decomposition property}, which is well-known to fail in general.
In terms of weighted Veronese subrings, the result may be reformulated as follows (Corollary \ref{CorBirationalVeroneseSubring}):
if $R$ is a positively graded polynomial ring, and  $R^{(M)}$ is a Veronese subring, where $M$ is divisible by the degrees of all the variables,
then the extension $k[R_M] \subset R^{(M)}$ is birational.

As a consequence, for a zero-dimensional monomial ideal $I$ with a single Rees valuation, 
the formula for $\nu_I$ collapses dramatically. Specifically, 
Theorem~\ref{Theoremonlyone} gives the formula
\[
\nu_I \;=\; |\mathcal{I}_n(\log \mathcal{A})|
\]
involving only the log matrix of the generators of $I$ whose exponent vectors lie on the 
unique compact facet of ${\rm NP}(I)$. In particular, for  
$I = (x_1^{a_1},\ldots,x_n^{a_n})$ we recover the simple 
identity
\[
\nu_I \;=\; \mathrm{lcm}(a_1,\ldots,a_n),
\]
which generalizes~\cite{Graffeo}*{Theorem~B} from $n=2$ to arbitrary number of variables (Corollary~\ref{Onlyone}).

\subsection*{The reduced special fiber ring of a monomial ideal}

The techniques developed for computing Behrend functions also lead to 
new structural results for the special fiber ring of an arbitrary 
monomial ideal,
not necessarily zero-dimensional.
In Section~\ref{SectionSpecialFiber},
we describe the minimal primes and the reduced structure of 
the special fiber ring $\cF(I)=\mathcal R(I) \otimes_R k$ for any monomial ideal $I$, in terms of the facets of 
the Newton polyhedron ${\rm NP}(I)$ 
(Theorem~\ref{ThmReducedSpecialFiberIdealDescription} and Corollary~\ref{Corollary6.12}). 
This generalizes work of Singla~\cite{Singla}, who treated the case of 
\emph{extremal} ideals, those equal to their minimal monomial 
reduction, in terms of the maximal compact faces of ${\rm NP}(I)$.

The two approaches differ in a substantive way that is worth pointing out. In Singla's setting, each constituent ideal is generated by the 
monomials whose exponents lie on a fixed maximal compact face of 
${\rm NP}(I)$; in fact, because the ideal $I$ is extremal, these exponents are 
precisely the vertices of that maximal compact face. Thus, each such 
ideal coincides with its own minimal reduction, so its special fiber 
ring is a polynomial ring and agrees with the corresponding toric ring. 
For arbitrary monomial ideals, neither of these properties holds: our 
constituent ideals $I_i$ are indexed by Rees valuations rather than 
by maximal compact faces, and the special fiber rings $\mathcal{F}(I_i)$ need not 
coincide with the corresponding toric rings (see Example~\ref{Ex7.15}). 
To handle this, we work directly with the special fiber rings 
$\cF(I_i)$ and show that even when they fail to be toric, 
their nilradicals can be controlled. In the zero-dimensional case, 
where Rees valuations correspond to maximal compact faces, 
our description specializes to a genuine generalization of Singla's 
formula from extremal to arbitrary zero-dimensional ideals, and raises the natural 
question of whether her formula admits an extension to monomial ideals of arbitrary dimension that are not extremal.

\subsection*{Independent interest}

We have aimed to make most of our commutative algebra results 
self-contained and to state them in their natural generality, even 
when this exceeds what is strictly needed for the  
applications to the Behrend function. We expect that several of these results, in particular 
the birationality theorem for weighted Veronese subalgebras 
(Theorem~\ref{ThmSameSubgroups}, Corollary~\ref{CorBirationalVeroneseSubring})
and the 
structure theorem for the reduced special fiber ring of a monomial 
ideal
(Theorem~\ref{ThmReducedSpecialFiberIdealDescription}, will be of independent interest.

\subsection*{Organization of the paper}
Section~\ref{SectionBehrend}
establishes the general length and Behrend function formulas in terms 
of Rees valuations. Section~\ref{SectionBehrendNormal} develops the theory in the normal case, recovering and 
extending the formulas of~\cite{Graffeo} through several worked 
examples. Section~\ref{SectionBehrendMonomial} is the technical core of the paper, where we prove the main formulas for arbitrary zero-dimensional 
monomial ideals via Ehrhart theory and log matrices. 
Section~\ref{SectionOneVal} establishes the 
birational Veronese theorem and derives the simplified formula for 
ideals with a single Rees valuation. Finally, 
Section~\ref{SectionSpecialFiber} treats the 
structure of the reduced special fiber ring of a monomial ideal.

\medskip

\section{The Behrend function}\label{SectionBehrend}

In this section, 
we establish the foundational formulas that govern 
all subsequent computations in the paper. 
Our point of departure is Behrend's formula for $\nu_I$.
 Let $R=\bC[x_1,\ldots,x_n]$ be a polynomial ring,  let $I\subset R$ be an  ideal, and write $A=R/I.$  Let  $G = \mathrm{gr}_I(R)$  be the associated graded ring of $I,$ and let $\mathcal M$ be the set of minimal primes of $G,$ which all have dimension $n.$ 
 Let $\nu_I$ denote the  Behrend function of  $X=\mathrm{Spec}(R/I)$.
 By \cite{Beh}, the Behrend function satisfies
 \begin{equation}\label{EqDefBehrendFunction} \nu_I= \sum_{\mathfrak{P}\in \mathcal M} (-1)^{\dim(A/\mathfrak{P} \cap A)} \, \lambda(G_\mathfrak{P}) \cdot {\rm Eu}([\mathfrak{P} \cap A]) \, , 
 \end{equation}
 where $\lambda$ denotes length and ${\rm Eu}$ is MacPherson's local Euler obstruction \cite{Mac74}, which associates to every  cycle on $X=\Spec(A)$, such as $[\mathfrak{P} \cap A],$ a constructible function $X \rightarrow \mathbb Z.$
  For a zero-dimensional ideal $I$, these local Euler obstructions are identically equal to one, so the entire content of the formula lies  in the lengths $\lambda(G_\fP)$ and the combinatorics of how the  minimal primes of~$G$ relate to the geometry of the blowup. 
Translating both quantities into the language of Rees rings and 
Rees valuations is the conceptual move that drives this paper.

To this end let, for simplicity, $R$ be a Noetherian ring that is normal, universally catenary, and locally analytically unramified, and let $I\subset R$ be an ideal of positive height. 
Consider the chain of inclusions
\[
R \,\subset\, \cR = R[It] \,\subset\, \overline{\cR}= \bigoplus_{j \geq 0} \,\overline{I^j}t^j \,\subset\, R[t],
\]
where $t$ is a variable, $\cR$ is the {\it Rees ring} of $I,$ and 
$\overline{\cR}$ denotes its integral closure. Since 
$G =\mathrm{gr}_I(R)= \cR/I\cR,$ 
we may identify the set of 
minimal primes of $G$ with the set $\cM$ of minimal primes of $I\cR.$ 
These are exactly the contractions of the minimal primes of 
$I \overline{\cR},$ and we write $\cN$ for this set. In turn, the primes
$Q \in \cN$ correspond to the {\it Rees valuations} of $I,$ the valuations 
$v_Q$ of the discrete valuation rings $\overline{\cR}_Q.$ Notice that the 
primes of $\cM$ and $\cN$ are precisely the centers of the Rees valuations
on $\cR$ and on $\overline{\cR},$ respectively. More commonly, the Rees valuations 
are defined as the restrictions of the valuations $v_Q$ to ${\rm Quot}(R),$
but we prefer to work with the rings $\cR$ and $\overline{\cR}$
rather than their dehomogenizations in order to take advantage of the grading.
Notice that $$\overline{\cR}=\bigcap_{Q\in \cN}{\overline{\cR}}_Q \cap R[t] \, ,$$
and that $It \not\subset Q,$ which gives $v_Q(It)=0$ and hence $v_Q(t)=-v_Q(I).$
One deduces that an element $\alpha \in R$ belongs to $\overline{I^j}$ if and only if $v_Q(\alpha) \geq j \cdot v_Q(I)$ for every $Q \in \cN;$ in other words,
the Rees valuations control the integral closures of $I$ and its powers.

Within this framework, our first main result 
(Proposition~\ref{PropLength}) is a general length formula 
expressing the length $\lambda(G_\fP)$  for any $\fP \in \cM$ as a 
sum of contributions from the Rees valuations weighted by residue 
field degrees. Combined with~\eqref{EqDefBehrendFunction}, this 
yields the formula
\[
\nu_I = \sum_{i,j} v_{i,j}(I) \cdot 
\bigl[\overline{\cR}/\fQ_{i,j} : \cR/\fP_i\bigr]
\]
of Theorem~\ref{TheoremFormulaBehrendWithFieldDegrees}.
Here, $\fP_i \in \cM,$ each $\fQ_{i,j} \in \cN$ lies over $\fP_i,$ and $v_{i,j}=v_{Q_{i,j}};$ moreover for an inclusion $A \subset B$  of domains, we denote by $[B:A]$ the degree of the field extension $\Quot(A)\subset \Quot(B).$
This formula is the central 
identity of the paper: the rest of our work is devoted to making 
it computationally effective by determining the Rees valuations 
$v_{i,j}$ and the residue field degrees explicitly.

We then turn to the case of homogeneous ideals, where the residue 
field degree admits a particularly transparent interpretation in 
terms of the multiplicity of a standard graded subalgebra. Building 
on a study of the reduced special fiber ring $\cF(I)/\sqrt{0},$ 
which we show coincides with the toric ring $k[I_d]$ when $I$ is 
integral over its initial-degree component $I_dR$ 
(Proposition~\ref{PropReducedFiberHomogeneousIdeal}), we obtain the 
closed formula
\[
\nu_I = d \cdot [\bC[\m_d] : \bC[I_d]] = \frac{d^n}{e(\bC[I_d])}
\]
of Theorem~\ref{ThmFormulaBehrendHomogeneous}, valid for all 
zero-dimensional homogeneous ideals $I \subset \bC[x_1,\ldots,x_n]$ 
integral over $I_dR.$ This class is considerably broader than the 
class of ideals generated in a single degree; for instance, it 
includes every ideal $I$ with $(x_1^d,\ldots,x_n^d) \subset I 
\subset (x_1^d,\ldots,x_n^d) + \m^{d+1},$ all of which share the 
same Behrend function $\nu_I = d^n$ (Remark~\ref{RemSqueeze}).

\begin{Proposition}\label{PropLength}
Let $R$ be a Noetherian universally catenary locally unramified normal ring and $I$ be an ideal
of positive height. By $\cR$ we denote the Rees ring of $I,$ by $\overline{\cR}$ the
integral closure of $\cR,$ and by $G=\cR/I \cR$ 
the associated graded ring of $I.$ Let
$\fP \in \Spec({\cR})$ be a minimal prime
of the ideal $I\cR.$ Let $\fQ_j \in \Spec({\overline{\cR}})$ be the primes lying over $\fP$ and $v_j$ be the corresponding 
Rees valuations of $I.$ Then 
\[\lambda(G_{\fP}) = \sum_{j}v_j(I) \cdot 
\left[(\overline{\cR}/\fQ_j):(\cR/\fP)\right]  . \]
\vspace{-.28cm}
\end{Proposition}
\begin{proof} As $\fP$ does not contain the irrelevant ideal of $\cR,$ we have 
$I\cR_{\fP}= f\cR_{\fP}$ for a regular element $f \in \cR_{\fP}.$ Since
$R_{\fP \cap R}$ is an analytically unramified local domain, $\overline \cR_{\fP}
=\overline{\cR_{\fP}}$ 
is a finite $\cR_{\fP}$-module, and therefore 
$\overline{\cR_{\fP}}/\cR_{\fP}$ has finite length over $\cR_{\fP}$ 
because $\dim \, \cR_{\fP} =1.$ 
Now, the isomorphism $\overline{\cR_{\fP }}/\cR_{\fP} \cong f\overline{\cR_{\fP}}/f\cR_{\fP}$ and the diagram

\begin{center}
\begin{tikzpicture}[scale=1]

  \node (top)    at (0,1.5)   {$\overline{\mathcal R_{\fP}}$};
  \node (left)   at (-2,0)    {$\mathcal R_{\fP}$};
  \node (right)  at (2,0)     {$I\overline{\mathcal R_{\fP}}= f\overline{\mathcal R_{\fP}}$};
  \node (bottom) at (0,-1.5)  {$I \mathcal R_{\fP} = f \mathcal R_{\fP}$};

  \draw[thick] (top) -- (left);
  \draw[thick] (bottom) -- (right);

  \draw[thick] (top) -- (right);
  \draw[thick] (bottom) -- (left);

\end{tikzpicture}

\end{center}

\noindent
show that $\lambda(G_{\fP})=\lambda_{\cR_{\fP}}(\cR_{\fP}/I \cR_{\fP})=
\lambda_{\cR_{\fP}}(\overline{\cR_{\fP}}/I \overline{\cR}).$ On the other hand,

\[
\begin{array}{rcll}
\lambda_{{\mathcal R}_\fP}\!\left(\overline{{\mathcal R}_\fP} /
                    I\overline{{\mathcal R}_{\fP}}\right) 
                &=& 
                \displaystyle\sum_{j}
                    \lambda_{{\mathcal R}_{\fP}}\!\left(\overline{{\mathcal R}_{\fQ_j}} /
                    I\overline{{\mathcal R}_{\fQ_j}}\right) \qquad \text{by the Chinese  Remainder Theorem}\\[2mm]

                &=& 
                   \displaystyle\sum_{j}
                    \lambda_{\overline{{\mathcal R}_{\fQ_j}}}\!\left(\overline{{\mathcal R}_{\fQ_j}} /
                    I\overline{{\mathcal R}_{\fQ_j}}\right) \cdot \left[\overline{{\mathcal R}_{\fQ_j}} /
                    \fQ_i\overline{{\mathcal R}_{\fQ_j}} \, : \, {\mathcal R}_\fP / \fP{\mathcal R}_\fP\right]\\[2mm]
                    
                      &=& 
                  \displaystyle\sum_{j} \ v_j(I)
                   \cdot \left[ \overline{\mathcal R}/\fQ_i : \mathcal R/\fP \right]\, .  
\end{array}
\] 
\end{proof}

The following  is a special case of Proposition \ref{PropLength},  essentially equivalent to \cite{Graffeo}*{Theorem D}.

\begin{Theorem}\label{TheoremFormulaBehrendWithFieldDegrees}
Let $R=\bC[x_1,\ldots,x_n]$ be a polynomial ring and $I\subset R$ be a zero-dimensional ideal.
Let $\fP_i$ be the minimal primes of $I\cR$,
$\fQ_{i,j}$ be the primes of $\overline{\cR}$ lying over $\fP_i$,
and $v_{i,j}$ be the Rees valuations of $I$ corresponding to $\fQ_{i,j}.$
Then 
\begin{equation}\label{EqTheoremFormulaBehrendWithFieldDegrees}
\nu_I = \sum_{i,j}  v_{i,j}(I)\cdot \left[\overline{\cR}/\fQ_{i,j}: \cR/\fP_i\right]. 
    \end{equation}
\end{Theorem}

\begin{proof}
This follows from Proposition \ref{PropLength} and formula \eqref{EqDefBehrendFunction}
\end{proof}

We now consider the problem of studying these lengths and blowup algebras in the general context of a homogeneous ideal of a graded ring.

\begin{Proposition}\label{PropReducedFiberHomogeneousIdeal}
Let $R$ be a non-negatively graded Noetherian algebra over a field $k.$
Let $I\subset R$ be a homogeneous ideal with initial degree $d$ and 
denote by $I_d$ its graded component of degree $d.$ Consider the 
ring $k[I_d] \subset R$ and the natural maps of standard graded 
$k$-algebras 
\[
\begin{tikzcd}[column sep=2em, ampersand replacement=\&]
k[I_d] \arrow[r, two heads, "\varphi_1"] \& \cF(I_dR) \arrow[r, "\varphi_2"] \& \cF(I)/\sqrt{0} \, ,
\end{tikzcd}
\]
where $\varphi_1$ is surjective.

If $I$ is integral over $I_dR,$ then also $\varphi_2$ is surjective, and if
in addition $R$ is reduced and positively graded, then $\varphi_1$ and $\varphi_2$ 
are isomorphisms.
\end{Proposition}
\begin{proof}
Let $\m$ denote the maximal homogeneous ideal of $R.$ 

To prove that $\varphi_2$ is surjective, it suffices to show that the image of every homogeneous element $f \in I \setminus I_d$ in $\cF(I)/\sqrt{0}$
is zero. 
By assumption, there is an equation of integrality in $R,$
$$
f^n+a_1f^n+\cdots + a_n = 0 \, ,
$$
where $n >0$ and $a_i \in (I_d R)^i$ are homogeneous of degree equal to $i \deg(f).$ 
As $i \deg(f) > id,$ we notice that 
$a_i \in \m(I_dR)^i\subset \m I^i.$ As a consequence we obtain
$f^n \in \m I^n.$ Therefore $f+\m I $ is nilpotent in $\cF(I)$, in other words, 
the image of $f+\m I $ in 
$ \cF(I)/\sqrt{0}$ is zero, as asserted.

To show that $\varphi_1$ and $\varphi_2$ are isomorphisms it suffices to prove that 
$\varphi:=\varphi_2 \circ \varphi_1$ is injective. Let $g$ be a homogeneous element of degree $t$ in 
the standard graded $k$-algebra $k[I_d].$
Suppose that $\varphi(g)=0.$ This means there exists $b \in \bN$ so that $g^b \in \m I^{bt}.$ 
However, $g^b$ is homogeneous of degree $btd$ in $R$, whereas $\m I^{bt}$ is generated in degrees at least $btd+1$ as $R$ is positively graded.
Thus $g^b=0$ and,
since $R$ is reduced, we conclude that $g=0,$ as required.
\end{proof}

\begin{Proposition}\label{PropLengthHomogeneousDomain}
    Let $R$ be a positively graded Noetherian domain over a field $k.$
Let $I\subset R$ be a zero-dimensional homogeneous ideal with initial degree $d,$
and assume that $I$ is integral over the ideal $I_dR.$
Let $G$ denote the associated graded ring of $I$.
Then  $G/\sqrt{0}$ is a domain and
$$
\lambda\left( G_{\sqrt{0}} \right)  = \frac{e_I(R)}{e(k[I_d])} \, .
$$
\end{Proposition}
\begin{proof}
We have $G/\sqrt{0}  \cong \cF(I)/\sqrt{0} \cong k[I_d],$
where the first isomorphism holds because $I$ is zero-dimensional and the second isomorphism 
follows from Proposition \ref{PropReducedFiberHomogeneousIdeal}. We conclude that
$G/\sqrt{0}$ is a domain and, in particular, $\sqrt{0}\subset G$ is 
the only minimal prime.
Now, using the the associativity formula, we obtain 
$$
e_I(R) = e(G)=\lambda\left( G_{\sqrt{0}} \right) \cdot e\left(G/\sqrt{0}\right) = 
\lambda\left( G_{\sqrt{0}} \right) \cdot e\left(k[I_d]\right) .
$$
This implies the desired conclusion.
\end{proof}

\begin{Theorem}\label{ThmFormulaBehrendHomogeneous}
Let $R=\bC[x_1,\ldots,x_n]$ be a standard graded polynomial ring and $\m$ be its maximal homogeneous ideal. 
Let $I\subset R$ be a  zero-dimensional homogeneous ideal with initial degree $d,$ and assume that $I$ is integral over $I_dR$.
Then 
$$
\nu_I = d \cdot \left[\bC[\m_d] : \bC [I_d]\right] =  \frac{d^n}{e(\bC[I_d])}.
$$
\end{Theorem}

\begin{proof}
Since $I$ integral over $I_dR,$  there is a homogeneous reduction of $I$ generated in degree $d$.
As $\mathbb C$ is infinite, such a reduction $J$ can be chosen to be generated by $n$ homogeneous elements of degree $d,$ and these elements form a regular sequence.
It follows that $e_I(R) = e_J(R) = d^n$.

By \cite{KPUfiber}*{Theorem 5.3} we have 
$$
e\left(\bC[I_d]\right) = \frac{d^{n-1}}{\left[ \bC[\m_d]:\bC[I_d]\right]} ,
$$
where the denominator is the degree of the rational map defined by the linear system $I_d$.
Now the conclusion follows from Proposition \ref{PropLengthHomogeneousDomain} and \eqref{EqDefBehrendFunction}.
\end{proof}

\begin{Remark}\label{RemSqueeze}
{\rm Theorem \ref{ThmFormulaBehrendHomogeneous} implies the following statement:
If $J\subset \mathbb C[x_1, \ldots, x_n]$ is a  zero-dimensional homogeneous ideal generated in a single degree $d,$
and  $I$ is any homogeneous ideal such that $J \subset I \subset \overline{J}$ and $J_d = I_d$, then $\nu_I = \nu_J$.
For example, all homogeneous ideals $I$ such that
$$(x_1^d, x_2^d, \ldots, x_n^d) \subset I \subset (x_1^d, x_2^d, \ldots, x_n^d) +\m^{d+1}$$ 
have the same Behrend function, namely, $\nu_I = d^{n}$.
}
\end{Remark}

\section{The Behrend function of normal ideals}\label{SectionBehrendNormal}

In this section, we illustrate the computation of the Behrend function for classes of normal ideals, 
for which 
formula \eqref{EqTheoremFormulaBehrendWithFieldDegrees}  becomes particularly effective.

\begin{Corollary}\label{CorFormulaBehrendWithFieldDegreesNormalCase}
Let $R=\bC[x_1,\ldots,x_n]$ be a polynomial ring and $I\subset R$ be a normal zero-dimensional ideal.
 Then 
\begin{equation}\label{EqCorFormulaBehrendWithFieldDegreesNormalCase}
\nu_I = \sum_{v}  v(I)
    \end{equation}
where $v$ ranges over the Rees valuations of $I.$
\end{Corollary}

\begin{proof}
        It follows  immediately from 
Theorem~\ref{TheoremFormulaBehrendWithFieldDegrees},
since normality of $I$ amounts to $\cR = \overline{\cR}.$
\end{proof}

Several classes of normal ideals are known.

\begin{Example}\label{ExampleNormalIdeals}
{\rm  Let $R=\bC[x_1,\ldots,x_n]$ be a polynomial ring and $I\subset R$ be a zero-dimensional integrally closed ideal. If $n=2,$ then $I$ is normal \cite{Zariski}*{Part II, Section 12}. 
Furthermore, $I$ is normal if the minimal number of generators $\mu(I)$ of $I$ is $\le n+2 \, $ \cite{EGHU}*{Theorem 2.2}, or $I$ is homogeneous with  $\mu(I)\le n+3 \,$ \cite{EGHU}*{Theorem 3.4}, or  $I$ is monomial with $\mu(I)\le n+4 \,$  \cite{AtMa}*{Theorem 4.1}. }
\end{Example}

When $I$ is a monomial ideal, 
the valuations  in 
 \eqref{EqCorFormulaBehrendWithFieldDegreesNormalCase} can be read off the Newton polyhedron ${\rm NP}(I)$, thus, 
convex geometry gives a simple 
 and general procedure to determine the Behrend function.
We illustrate this calculation for $n = 2$ variables, 
when the formula becomes so explicit that it can be expressed in terms of the exponent vectors of the minimal generators of $I$ corresponding to the vertices of ${\rm NP}(I)$.
In the next section, we will give general formulas  for the Behrend function of monomial ideals.

\begin{Example}[Behrend function of integrally closed monomial ideals in two variables, I]\label{ExampleBehrendNormalTwoVariablesI}
{\rm 
Integrally closed ideals of $R = \bC[x,y]$ are normal, cf. Example \ref{ExampleNormalIdeals}.
Thus, a zero-dimensional monomial ideal $I \subseteq R$ is normal if and only if
$$
I = \overline{\left(
x^{a_0}, x^{a_1}y^{b_1}, \dots, x^{a_{r-1}}y^{b_{r-1}}, y^{b_r}\right)
}
$$
for some  integers $a_0 >\cdots> a_{r-1}> 0 =:a_r$ and $b_0 := 0 < b_1 < \cdots < b_r$ such that
the slopes $\frac{b_{i-1}-b_{i}}{a_{i-1}-a_{i}}$ are decreasing in $i$.
Equivalently, $I$ is the monomial ideal whose exponent vectors lie in the Newton polyhedron $${\rm NP}(I) = \mathrm{conv}\big((a_0,0),\ldots,(0,b_r)\big)+\mathbb{R} _{\geq 0}^2\,.$$
The Rees valuations $v_1, \ldots, v_r$ of $I$ correspond to the line segments that are  compact facets of   ${\rm NP}(I)$.
Denoting by $(z_1,z_2)$ the coordinates of the plane containing ${\rm NP}(I)$,
the line through $(a_{i-1},b_{i-1})$ and $(a_i,b_i)$ 
has equation  $ (b_i-b_{i-1})z_1+(a_{i-1}-a_{i})z_2=a_{i-1}b_i-b_{i-1}a_i$,
and the value of the ideal satisfies
$
v_i(I) \gcd(b_i-b_{i-1},a_{i-1}-a_i)= a_{i-1}b_i-b_{i-1}a_i.
$ 
By \eqref{EqCorFormulaBehrendWithFieldDegreesNormalCase}
we obtain
\begin{equation}
\nu_I=  \sum_{i=1}^r \frac{a_{i-1}b_i-b_{i-1}a_i}{\gcd(b_i-b_{i-1},a_{i-1}-a_i)}.
\end{equation}
When $r=1$, this formula recovers \cite{Graffeo}*{Theorem B}.
}
\end{Example}

The argument of Example \ref{ExampleBehrendNormalTwoVariablesI} holds
for normal monomial ideals in any number $n$ of variables.
Observe that, if $n \geq 3$, 
an integrally closed monomial ideal may not be normal,
cf. Examples \ref{ExFails}, \ref{E5.10}.
Nevertheless, we will show 
in Corollary \ref{Corn3}
that,
 somewhat surprisingly,
this  argument still works for (possibly non-normal) integrally closed ideals in $n=3$ variables.

We point out that,
for monomial ideals, 
the calculation of the Newton polyhedron and of the Rees valuations can be performed with convex geometry softwares such as Normaliz
\cite{Normaliz}.

An alternative approach,  available for integrally closed ideals in 2 variables
and avoiding the explicit use of the Newton polyhedron,
is based on  Zariski's factorization theorem.

\begin{Example}[Behrend function of integrally closed monomial ideals in two variables, II]\label{ExampleBehrendNormalTwoVariablesII}
{\rm 
Zariski proved that an integrally closed $\m$-primary ideal $I \subset R=  \mathbb{C}[x,y]$
 factors uniquely (up to the order of the factors) as a product $I = I_1\cdots I_s$ of integrally closed 
 $\m$-primary ideals $I_i \subset \mathbb{C}[x,y]$ that are \emph{simple}, i.e., are
 not the product of two proper ideals. An equivalent characterization of the ideals
 $I_i$ is that each $I_i$ is an $\m$-primary ideal defined by one Rees valuation $v_i,$ 
$$I_i = \big( f \in \mathbb{C}[x,y] \;\big|\; v_i(f) \geq v_i(I_i) \big).$$
The set of Rees valuations of $I$ is precisely $\{v_1, \ldots, v_s\}$
by \cite{LipmanProximity}*{Proposition 4.4}.
Clearly, valuations respect products of ideals, that is,
$v_i(I) = \sum_{j=1}^s v_i(I_j)$.
By Corollary \ref{CorFormulaBehrendWithFieldDegreesNormalCase}, we obtain the formula
\begin{equation}\label{EqBehrendNormalIdealTwoVariables}
\nu_I = \sum_{i,j=1}^s v_i(I_j).
    \end{equation}
    }
\end{Example}
    
This approach is particularly convenient when a factorization of $I$ is known a priori. 
As a test case, we illustrate it  for \emph{towers}, a class of ideals introduced and  extensively  investigated in \cite{Graffeo}*{Sections 3 and 4}.

\begin{Example}[Behrend function of towers]\label{ExBehrendTowers}
{\rm 
An $\m$-primary ideal $I \subset \bC\llbracket x,y\rrbracket$ is called a tower \cite{Graffeo}*{Definition 3.4} if there exist positive integers $\ell_1 < \cdots < \ell_s$ such that, 
up to a $\mathbb{C}$-automorphism of $\bC\llbracket x,y\rrbracket$, we have
$$
I = \prod_{i=1}^s I_i
\quad
\text{where}
\quad
I_i=\big(x, y^{\ell_i}\big).
$$
Each $I_i $ is a simple integrally closed $\m$-primary ideal, defined by the monomial valuation $v_i$ such that $v_i(x) = \ell_i, v_i(y) =1$.
We observe that
$
v_i(I_j) = \min \{ v_i(x), v_i\big(y^{\ell_j}\big)\} = \min \{ \ell_i, \ell_j\}$,
thus,  by  \eqref{EqBehrendNormalIdealTwoVariables}, 
we obtain
$$
\nu_I = \sum_{1\le i,j\le s}  \min \{ \ell_i, \ell_j\} = \sum_{i=1}^s(2s+1-2i)\ell_i.
$$
This recovers  \cite{Graffeo}*{Theorem  3.13}, since the latter sum is equal to
$\sum_{i=1}^s\sum_{j=1}^i \ell_j + \sum_{i=1}^{s-1}\ell_i(s-i).$
A special case is that of {\it complete towers}, when $\ell_i=i$ for all $i$: then the formula simplifies to $\nu_I = \binom{s+1}{3}+\binom{s+2}{3},$
and recovers  \cite{Graffeo}*{Theorems  3.11}.
}
\end{Example}

\begin{Example}[Products of complete towers, I]\label{ExProdTowersI}
{\rm 
Let $K =IJ \subset \bC\llbracket x,y\rrbracket$ be the product of two ideals 
$$
I = \prod_{i=1}^sI_i \quad \text{where}\quad I_i=\big(x,y^i\big) 
\quad \text{and} \quad 
J = \prod_{j=1}^rJ_j \quad \text{where}\quad J_j=\big(x^j,y\big),
$$
for some positive integers $r,s$, possibly after a $\mathbb{C}$-automorphism of $\bC\llbracket x,y\rrbracket$.
As in Example \ref{ExBehrendTowers}, $I_i$ and $J_j$ are simple integrally closed $\m$-primary ideals,
each defined by one monomial Rees valuation, respectively, 
$v_i$ and $w_j$, such that
$$
v_i(x) = i, \, v_i(y) = 1 \quad \text{and} \quad 
w_j(x) = 1,\, w_j(y)=j.
$$
Observe that $v_1=w_1$, so the set of distinct Rees valuations of $K$ is $\{v_1, \ldots, v_s, w_2, \ldots, w_r\}$.
We compute 
$$
v_i(J_j) = \min\{ v_i(x^j), v_i(y)\} = 1,\quad
w_j(I_i) = \min\{ w_j(x), w_j(y)^i\} = 1,
$$
and, applying  \eqref{EqBehrendNormalIdealTwoVariables},
we recover the formula of \cite{Graffeo}*{Theorem 3.17 (1)}:
\begin{align*}
\nu_{K} &=  \sum_{i=1}^s \sum_{h=1}^s  v_i(I_h) 
+
\sum_{i=1}^s \sum_{j=1}^r v_i(J_j) 
+ \sum_{j=2}^r \sum_{i=1}^s w_j(I_i)
+ 
\sum_{j=2}^r  \sum_{k=1}^r w_j(J_k)
\\
&
= \nu_I+ \nu_J - \sum_{k=1}^r w_1(J_k) 
+ 
\sum_{i=1}^s \sum_{j=1}^r (v_i(J_j) +w_j(I_i)) - \sum_{i=1}^s w_1(I_i)
\\
&
= \nu_I + \nu_J+ 2rs  -r-s. 
\end{align*}
}    
\end{Example}

\begin{Example}[Products of complete towers, II]\label{ExProdTowersII}
{\rm 
Let $K =IJ \subset \bC\llbracket x,y\rrbracket$ be the product of two ideals 
$$
I = \prod_{i=1}^sI_i \quad \text{where}\quad I_i=\big(x,y^i\big) 
\quad \text{and} \quad 
J = \prod_{j=1}^rJ_j \quad \text{where}\quad J_j=\big(x+g,y^j\big),
$$
for some positive integers $r,s$, 
and some $g \in y \, \bC\llbracket y\rrbracket$,
possibly after a $\mathbb{C}$-automorphism of $\bC\llbracket x,y\rrbracket;$ in fact we may assume that
$g \in y \, \mathbb C [y]$ with ${\rm{deg}} \, g <r.$
Let $d = o(g)$ be the $(y)$-adic order of $g.$ Clearly $d < r.$
We may also assume
that $d < s; $ for otherwise changing the variable $x$ to $x':=x+g$ we have 
$I_i=(x'-g, y^i)=(x',y^i)$ and so $K$ is monomial in the variables $x',y.$
In either case $\nu_K$ can be computed as in Example \ref{ExBehrendTowers}.

As before, $I_i$ and $J_j$ are simple integrally closed $\m$-primary ideals,
each defined by one Rees valuation, denoted by 
$v_1,\ldots, v_s, w_1,\ldots, w_r.$
Unlike the previous examples, $K$ is not a monomial ideal.
However,
each Rees valuation is monomial in some system of local coordinates:
$v_i$ is monomial  in $x,y,$ determined by 
$v_i(x) = i, \, v_i(y) = 1;$ 
if $j \leq d$, then $J_j = I_j$ and $w_j = v_j;$
if $j \geq d+1$, then 
$w_j$ is monomial  in $x+g, y$, determined by
$
w_j(x+g) = j,\, w_j(y)=1;
$
if $j \leq d$, then $J_j = I_j$ and $w_j = v_j.$ 
Thus,
the set of distinct Rees valuations of $K$ is  $\{v_1, \ldots, v_s, w_{d+1}, \ldots, w_r\}$.
These valuations are distinct for
the following reason: one has $v_i(x+g)={\rm{min}}\{v_i(x),v_i(g)\}$ because $v_i$
is a monomial valuation in $x,y$ and $g$ is a polynomial, hence $v_i(x+g)=
{\rm{min}}\{i,d\},$ whereas $w_j(x+g)=j.$

Now we compute 
$$
v_i(J_j) = \min \{ v_i(x+g),v_i(y^j)\} =\min \{ v_i(x),v_i(g),v_i(y^j)\}
= \min \{ i, d,j\},
$$
$$
w_j(I_i) = \min \{ w_j(x),w_j(y^i)\} =\min \{ w_j(x+g), w_j(g), w_j(y^i)\}
= \min \{ j, d,i\},
$$
where $w_j(x)=w_j(x+g-g)= \min\{w_j(x+g),w_j(g)\}$ since $w_j$ is a monomial valuation in $x+g,y$ and 
$g$ is a polynomial in $y$.

Applying  \eqref{EqBehrendNormalIdealTwoVariables}, we obtain the formula
\begin{align*}
\nu_{K}  &= \sum_{h=1}^s\sum_{i=1}^s v_i(I_h) +\sum_{i=1}^s\sum_{j=1}^r v_i(J_j)+
\sum_{j=d+1}^r\sum_{i=1}^s w_j(I_i) +
\sum_{j=d+1}^r\sum_{k=1}^r w_j(J_k) 
\\
&= \nu_I +\sum_{i=1}^s\sum_{j=1}^r v_i(J_j)+
\sum_{j=d+1}^r\sum_{i=1}^s w_j(I_i) +
\nu_J -
\sum_{j=1}^d\sum_{k=1}^r w_j(J_k) 
\\ 
& =\nu_I+\nu_J + \sum_{i=1}^s\sum_{j=1}^r \min\{i,j,d\} + \sum_{i=1}^s\sum_{j=d+1}^r \min\{i,j,d\} -
\sum_{j=1}^d\sum_{k=1}^r \min\{j,k\}
\\
& =\nu_I+\nu_J
+ \sum_{i=1}^s\sum_{j=d+1}^r \min\{i,j,d\} 
+ \sum_{i=d+1}^s\sum_{j=1}^r \min\{i,j,d\} 
+ 2\sum_{i=d+1}^s\sum_{j=d+1}^r \min\{i,j,d\}
\\
& =\nu_I+\nu_J
+ \sum_{j=d+1}^r\sum_{i=1}^s \min\{i,d\} 
+ \sum_{i=d+1}^s\sum_{j=1}^r  \min\{j,d\} 
+ 2\sum_{i=d+1}^s\sum_{j=d+1}^r d
\\
& =\nu_I+\nu_J + (r-d)\left( ds-\binom{d}{2}\right)+ 
(s-d)\left(dr-\binom{d}{2}\right) + 2 (s-d)(r-d)d.
\end{align*}
 This corrects  an error in the formula of \cite{Graffeo}*{Theorem 3.17(2)}.  
}    
\end{Example}

\vspace{.2cm}

\section{The Behrend function of a monomial ideal}\label{SectionBehrendMonomial}

In this section, we provide several general formulas for the Behrend function  $\nu_I$ for an arbitrary zero-dimensional monomial ideal $I \subset \mathbb C[x_1, \ldots, x_n]$.
In light of Theorem~\ref{TheoremFormulaBehrendWithFieldDegrees}, computing $\nu_I$
amounts to
determining, with notation as in the theorem,
\begin{itemize}
    \item the set $\mathcal{M}=\{\fP_i \}$ of  minimal primes of $I\cR$,
    \item the set $\mathcal{N}=\{\fQ_{i,j} \} $ of  minimal primes of $I\overline{\cR},$ where
    $\fQ_{i,j}$ lies over $\fP_i$, and the values $v_{i,j}(I)$ of $I$ with respect to the corresponding Rees valuations,
    \item the degrees of the residue field extensions $k(\fP_i) \subset k(\fQ_{i,j}).$ 
\end{itemize}
In this section, we solve these three problems for any monomial ideals, not necessarily zero-dimensional ones.

More generally, we work over an arbitrary field $k$, as most results hold in this generality.
Let $R=k[x_1, \ldots, x_n]$ denote a polynomial ring,
$\m$  the maximal homogeneous ideal of $R,$ and $I$  a monomial ideal.
As before, $\cR = R[It] \subset R[t]$ will denote the Rees ring of $I.$

We begin with a general  fact; we are grateful to Mircea Mustață for pointing
 out a simple argument to us.

\begin{Lemma}\label{L1-1} Let $A \subset B$ be an integral extension of monomial subalgebras of
a polynomial ring  over a field. 
The map $\Spec(B) \twoheadrightarrow
\Spec(A)$ induces a one-to-one correspondence between the monomial primes of
$B$ and the monomial primes of $A.$
\end{Lemma} 
\begin{proof} Owing to the multigrading, every monomial $b \in B$ satisfies an equation of integrality of the form $b^n+a=0$ with $a\in A.$ Now, if $\fQ$ is a monomial prime ideal of $B,$ then 
for every monomial $b \in \fQ$ there exists an integer $n >0,$ so that $b^n \in A \cap \fQ= :\fP.$
Thus a monomial of $B$ belongs to $\fQ$ if and only if it has a power that belongs to $\fP.$
In particular, $\fQ$ is uniquely determined by $\fP.$
\end{proof}

\begin{Proposition}\label{1-1}
Let $R=k[x_1,\ldots,x_n]$ be a polynomial ring and $I\subset R$ be a monomial ideal. 
Then 
the map $\Spec(\overline{\cR}) \twoheadrightarrow \Spec(\cR)$ induces a one-to-one 
correspondence between the minimal primes $\fQ_i$ of $I \overline{\cR}$ and the minimal 
primes $\fP_i$ of $I \cR.$ 

In particular, $$\lambda(G_{\fP_i})=v_i(I) \cdot  \left[\overline{\cR}/\fQ_{i}: \cR/\fP_i\right],$$
where $v_i$ is the Rees valuation of $I$ corresponding to $\fQ_i.$
If, in addition, $k=\mathbb C$ and $I$ is zero-dimensional, then 
\[\nu_I = \sum_{i}  v_{i}(I)\cdot \left[\overline{\cR}/\fQ_{i}: \cR/\fP_i\right]. \]

\end{Proposition}
\begin{proof}For the assertion about the minimal primes we use Lemma~\ref{L1-1}; notice that
$\cR \subset \overline{\cR}$ is an integral extension of monomial subalgebras of the
polynomial ring $k[x_1, \ldots, x_n,t]$ and that $I\cR$ and $I \overline \cR$ are monomial ideals. The statement about the length of $G_{\fP_i}$ follows from Proposition~\ref{PropLength}, and the 
formula for the Behrend function is then
a consequence of Theorem~\ref{TheoremFormulaBehrendWithFieldDegrees}.
\end{proof}

\begin{Remark}{\rm 
The one-to-one correspondence established in Proposition~\ref{1-1} generalizes a result of
Graffeo and Ricolfi \cite{Graffeo}*{Theorem C}, who proved 
an equivalent statement 
for any integrally closed monomial ideal $I$ in a polynomial ring in two variables. They
show that there is a one-to-one correspondence between the minimal primes of the associated
graded ring $G$ and the distinct simple integrally closed ideals that appear in a factorization of 
$I$ \cite{Graffeo}*{Proposition 4.4}.
These factors, in turn, correspond to the Rees valuations of $I$ according to \cite{LipmanProximity}*{Proposition 4.4}.
}
\end{Remark}

Proposition~\ref{1-1} accomplishes, in a way, the task of identifying the minimal 
primes of the associated graded ring, equivalently, the minimal primes of $I\cR,$ 
and the values of $I$ with respect to the Rees valuations. Indeed, for a monomial
ideal 
the minimal primes of $I{\cR}$ correspond to
the Rees valuations of $I,$ and 
these are in one-to-one
correspondence with the non-coordinate supporting hyperplanes of the Newton polyhedron ${\rm NP}(I).$ 
Explicitly, if $\ell_1{y_{1}} + \cdots + \ell_n{y_{n}}=M$
is the equation of such a hyperplane, 
where $\ell_1, \ldots, \ell_n, M $ are non-negative integers with  $\gcd(\ell_1, \ldots, \ell_n, M)=1$,
then the corresponding Rees valuation 
is the monomial valuation giving the variables $x_{i}$ value $\ell_{i};$ 
furthermore, the value of $I$ is $M.$

The remainder of this section is devoted to the 
more difficult problem of determining
the field degrees
$\left[\overline{\cR}/\fQ_i: \cR/\fP_i\right]$ in the formula of Proposition~\ref{1-1}.

\begin{Lemma}\label{LemmaVeronese} Let $B$ be a $\mathbb Z$-graded domain with $B_1 \neq 0.$ Then  $\,[B:B^{(d)}]=d \, $
for all  $d\in \mathbb Z_{>0}.$     
\end{Lemma}
\begin{proof}Write $A=B^{(d)}$, let 
$K$ be the quotient field of $A,$ and choose an element $x \neq 0$ in $B_1.$ We claim that 
$\{x^{d-i}|1 \leq i \leq d\}$ is a 
$K$-basis of $K\otimes_AB.$ Indeed, if $j$ is any integer, then $B_j x^i\subset A$ for some
$1 \leq i \leq d,$ so $B_j \subset (Ax^{-d})x^{d-i}
\subset Kx^{d-i}.$ This shows that
 $K\otimes_AB= \oplus_{1\leq i\leq d} Kx^{d-i}$.
\end{proof}

\begin{NotationDiscussion}\label{Dgrading}{\rm
In this section, we use the notation $M_\bullet$ to denote the graded components of 
graded modules.

Let $R=k[x_1,\ldots,x_n]$ be a polynomial ring over a field $k$ and $I\subset R$  a $\cR$ and $\overline{\cR},$ let $\fQ$ be the unique prime of $\overline{\cR}$ lying over $\fP,$ let $v$ be the corresponding Rees valuation of $I,$ and write $d=v(I).$ The valuation $v$ defines
a $\mathbb Z$-grading on the polynomial ring $R[t]=k[x_1, \ldots, x_n,t]$ giving
each variable degree equal to its value.
As explained at the beginning of Section \ref{SectionBehrend}, $v(t)=-v(I),$ 
so ${\rm{deg}}(t)=-d<0.$ Furthermore, since the ideal $I$ is monomial, 
${\rm{indeg}}(I) = d$ and in fact ${\rm{indeg}}(I^j) = {\rm{indeg}}(\overline{I^j}) = j \cdot d.$

The grading on $R[t]$ induces a grading on the monomial subalgebras
$\cR$ and $\overline{\cR},$
which is non-negative because these rings are contained in $\overline{\cR}_{\fQ}.$
From the definition of this grading by means of the monomial valuation $v$ one sees that the centers of $v$ on 
$\cR$ and on $\overline{\cR}$ are 
the irrelevant ideals of $\cR$ and $\overline{\cR},$ respectively. So
$$\fP = \cR_{>0}  \hbox{ \ \ and \ \ }  \fQ = \overline{\cR}_{>0},$$
which gives
$$\cR/\fP=\cR_0  \hbox{ \ \ and \ \ }  \overline{\cR}/\fQ=\overline{\cR}_0.$$

To describe the rings $\cR_0$ and $\overline{\cR}_0$ explicitly, we assume that
the center of $v$ on $R$ is the maximal homogeneous ideal, $\cP \cap R=\m,$ or equivalently that $R_0=k.$ This
assumption is no restriction for our purposes, as we will see in Discussion \ref{Dnozerodim}, and 
it is satisfied for every Rees valuation if (and only if) $I$ is zero dimensional. 
With this assumption on $v$ we consider the monomial $k$-subalgebras of $R[t],$
\[ A':=k[I_dt] \subset B':= k[(\overline I)_dt] \subset C':= \underset{j \ge 0}{\oplus}(\overline{I^j})_{jd}t^j \, . 
\]
Because $R_0=k$ and ${\rm{indeg}}(I^j) = {\rm{indeg}}(\overline{I^j}) = j \cdot d= -{\rm{deg}}(t^j),$ 
it immediately follows that $$\cR_0=A' \hbox{ \ \ and \ \ } {\overline{\cR}_0=C'}.$$
An isomorphic copy of the extensions $A'\subset  B' \subset C'$
is given by the monomial $k$-subalgebras of $R,$
\[ A:=k[I_d] \subset B:= k[(\overline I)_d] \subset C:= \underset{j \ge 0}{\oplus}(\overline{I^j})_{jd} \, , 
\]
which have the advantage that their dimension is maximal in the ambient polynomial ring $R,$ as we will see. 

Setting $\deg(x_i)=0$ and $\deg(t)=1$ instead defines a nonnegative grading on $R[t],$ which
we call {\it filtration grading} as it induces the filtration grading on the Rees algebra and 
its integral closure. This grading induces a positive grading on $A',$ $B',$ $C',$ and
on $A,$ $B,$ $C$ via the above isomorphism. In this grading
$A'$, $B',$ $A,$ $B$ become standard graded $k$-algebras. 

Notice that 
the extension $\cR \subset \overline{\cR}$ is finite and birational. So after passing to degree zero 
components, 
the extension $A' \subset C'$
is still finite, but not necessarily birational, an issue that will occupy us in this and the next section.
Obviously, the same holds for $A \subset C.$

}

\end{NotationDiscussion}

We summarize this discussion in the following result:

\begin{Proposition}\label{LIso} We use the setting of Discussion~$\ref{Dgrading}$
including the assumption $\fP \cap R=\m.$
The $\mathbb Z$-grading defined there induces a non-negative grading
on $\overline{\cR}$ and on $\cR$ so that
\begin{itemize}
\item $\cR_{>0}=\fP \,$ and $\, \cR_0= k[I_dt] =A' \cong A \, ,$ 
    in particular, $\, {\cR}/\fP = A' \cong A$;
    \item $\overline{\cR}_{>0}=\fQ \,$ and $\, \overline{\cR}_0=\underset{j \ge 0}{\oplus}(\overline{I^j})_{jd} \, t^j =C' \cong C \, ,$ 
    in particular, $\, \overline{\cR}/\fQ = C' \cong C$.
\end{itemize}
\end{Proposition}
\vspace{.2cm}

\begin{NotationDiscussion}\label{log2}{\rm 
We keep using the setting of Discussion~$\ref{Dgrading}$
including the assumption $\fP \cap R=\m.$
Proposition ~\ref{LIso} also implies the (known) facts that
the rings $A, A', C, C'$ have dimension $n$, since $\cR/\fP$ and $\overline{\cR}/\fQ$ do,
and that $C, C'$ are normal, since $C'$ is a direct summand of a normal domain. 

The supporting hyperplane $H$ of $N(I)$ corresponding to $v$
has the equation $$v(x_1) \, y_1+ \cdots +v(x_n) \, y_n=d,$$ where all the 
coefficients $v(x_i)$ are positive since $\fP \cap R=\m.$ Therefore
$\mathcal P :=H \cap N(I)$ is bounded, hence a polytope. This polytope is also the convex hull of the exponent vectors of the finitely many
monomials in $I_d.$
Consider the $j$-th dilation $j \mathcal P$,
then $j \mathcal P \cap \mathbb Z^n$
is the set of exponent vectors of the monomials in $(\overline{I^j})_{jd}.$ 
Likewise,
setting
$\mathcal P':= \mathcal P \times \{1\} \subset \mathbb R^n \times \mathbb R=\mathbb R^{n+1},$
then $j \mathcal P' \cap \mathbb Z^{n+1}$ is the set of exponent vectors of the monomials in $\, (\overline{I^j})_{jd}t^j.$ 
Thus, $C'$ is the Ehrhart ring of $\mathcal P$, while $B'$ is
the polytopal ring of $\mathcal P$, sometimes also called simply the toric ring of $\cP$ (see, for instance, \cite{Villa2}*{Chapter 9}). 
By \cite{StanleyEC1}*{Proposition 4.6.30}, the leading coefficient of the Ehrhart polynomial, the Hilbert polynomial of $C',$ is equal to $\rm{vol}(\mathcal P),$ the volume of $\mathcal P$ (see also \cite{Villa2}*{Lemma 9.3.7}).

Now, let $\mathcal A=\{\alpha_1, \ldots, \alpha_q\}$ be a monomial basis of the vector space $A_1,$ which is simply the collection 
of all monomials $\alpha \in I$ with $v(\alpha)=v(I).$  In addition, let $\mathcal A'=\{\alpha_1t, \ldots, \alpha_qt\}$  be the corresponding monomial basis of the vector space $A_1'=A_1t,$ and consider the set $\mathcal A''=\{\frac{\alpha_2}{\alpha_1}, \ldots, \frac{\alpha_q}{\alpha_1}\}.$ 

Given a
finite set $\mathcal U$ of monomials in the Laurent polynomial ring $k[x_1^{\pm1},\dots,x_n^{\pm1}]$,
we denote by ${\log}(\mathcal U)$ the {\it log  matrix} of $\mathcal U,$ the matrix whose columns are the
exponent vectors of the monomials in $\mathcal U$ (in any order). 
For an ideal $\mathfrak a$ of $\mathbb Z$,
we write  $| \mathfrak a |$ for the nonnegative integer generating $\mathfrak a, $
equivalently, the greatest common divisor of any finite generating set of $\mathfrak a.$ 

Notice that $\, {\rank} ({\log} (\mathcal U))={\rm{trdeg}}_k \, k[\mathcal U]={\dim} \, k[\mathcal U]\, $ (see, for instance \cite{Villa2}*{Corollary 8.2.21}). 
Therefore, $\, {\rank} ({\log} (\mathcal A))={\rank} ({\log} (\mathcal A'))=n, $ because $\, \dim \, A'=\dim \, A=n\, $ by Lemma~\ref{LIso}, for instance. Moreover, $\, {\rank} ({\log} (\mathcal A''))={\rank} ({\log} (\mathcal A))-1=n-1, $ because $k[\mathcal A'']$ is a dehomogenization of $A.$ In terms of matrices, if $u_1, \ldots, u_q$ denote the exponent vectors of
$\alpha_1, \ldots, \alpha_q, $ then ${\log} (\mathcal A)$ is the $n \times q$ matrix 
with columns $u_1, \ldots, u_q,$ ${\log} (\mathcal A'')$ is the $n \times (q-1)$ matrix with
columns $u_2-u_1, \ldots, u_q-u_1, $ and ${\log} (\mathcal A')$ is the $(n+1) \times q$ matrix
obtained from ${\log} (\mathcal A)$ by appending a row of ones. Notice that elementary 
column operations over $\mathbb Z$ transform ${\log} (\mathcal A')$ into the matrix
\vspace{-.1cm}
\begin{equation*} 
\begin{pmatrix}
{\log} (\mathcal A'')& u_q  \\
0 & 1 
\end{pmatrix} \ .
\end{equation*}
Thus, once more, we see that ${\rank} ({\log} (\mathcal A''))={\rank} ({\log} (\mathcal A'))-1$,
and that the ideals of maximal minors of the two log matrices are the same, 
$\mathcal{I}_{n-1}({\log} (\mathcal A''))= \mathcal{I}_n({\log} (\mathcal A')).$
}
\end{NotationDiscussion}

We are now ready to state one of our main technical results. The second and third 
equality in part (b)  are an immediate consequence of \cite{Villa2}*{Theorem 9.3.25(a)}
and \cite{EMV}*{Proposition 3.5}.


\begin{Theorem}\label{Lfielddegree} 
We use the setting of Discussion~$\ref{Dgrading}$ and Discussion~$\ref{log2}$
including the assumption $\fP \cap R=\m.$
\begin{enumerate}[{$(a)$}]
\item
$ [(\overline{\cR}/\fQ):(\cR/\fP)]=[C:A] = \frac{(n-1)! \ {\rm{vol}}(\mathcal P)}{e(A)} \, . $
\vspace{.2cm}
\item $ [(\overline{\cR}/\fQ):(\cR/\fP)]=[C:A] =  | \mathcal{I}_{n}({\rm log}(\mathcal A'))|= | \mathcal{I}_{n-1}({\rm log}(\mathcal A'')) | \, . $
\vspace{.2cm}
\item
In particular, 
$$\lambda(G_{\fP})= v(I)  \cdot \frac{(n-1)! \ {\rm{vol}}(\mathcal P)}{e(A)} =v(I)  \cdot | \mathcal{I}_n({\rm log}(\mathcal A'))| = v(I)  \cdot | \mathcal{I}_{n-1}({\rm log}(\mathcal A'')) |   \, .$$
\end{enumerate}
\end{Theorem}
\begin{proof}The first equalities in (a) and (b) hold by Proposition~\ref{LIso}. To 
prove the second equalities, recall from Discussion~\ref{Dgrading} and Discussion~\ref{log2}
that $A$ is a standard graded $k$-algebra of dimension $n$ and that the ring extension $A \subset C$ is homogeneous and module finite. Hence, $C$ has a Hilbert polynomial, necessarily of degree $n-1,$ and $e(C)=[C:A] \cdot e(A).$ Thus, $[C:A]=\frac{e(C)}{e(A)}=\frac{e(C')}{e(A')}.$

Now the second equality in (a) follows because $e(C')=(n-1)! \, {\rm vol}(\mathcal P),$
as explained in Discussion~\ref{log2}. As to the 
second equality in (b), recall that 
$\rank(\mathcal A')=n$ by 
Discussion \ref{log2}. Therefore \cite{Villa2}*{Theorem 9.3.25(a)} shows that 
$\frac{e(C')}{e(A')}= | \mathcal{I}_{n}({\rm log}(\mathcal A')|,$ as required. The 
third equality in (b) holds because $\mathcal{I}_n({\log}(\mathcal A'))=\mathcal{I}_{n-1}({\log}(\mathcal A''))$
according to Discussion~\ref{log2}.

Finally, item (c) follows from Proposition~\ref{PropLength} and parts (a) and (b).
\end{proof}

Next, we are going to prove formulas for the ratio $\frac{e(C)}{e(A)}$ that involve directly the
log matrix of $\mathcal{A}$, rather than $\mathcal{A'}$ or $\mathcal{A''}.$ This will
require a better understanding of the Ehrhart ring $C' \cong C.$

\begin{Discussion}\label{DiscussionBiratVerInd}{\rm 
In general,
the algebra $C'$ in Discussion~\ref{Dgrading}
is not standard graded.
However,
it is known   that the $\cR$-module $\overline{\cR}$ is generated in filtration 
degrees $\leq n-1$ (see, for instance, \cite{PUVV}*{Theorem 2.5}, \cite{RRV}*{Proposition 3.1}),
and this implies that the $A'$-module $C'$ is  
generated in degrees $\leq n-1$.
In fact, more is true:
since $C'$ is the Ehrhart ring of $\mathcal P$ and $B'$ is
the polytopal ring of $\mathcal P,$ the $B'$-module $C'$ is
generated in degrees $\, \leq \dim \, \mathcal P -1 =n-2 \,$ (see
\cite{BG}*{Theorem 2.52}), 
and the same is true for the $B$-module $C.$ 
When $\, n \leq 3,$ we obtain $B=C.$

Since 
$C$ is generated in degrees $\leq n-2$ as a module over the standard graded algebra $B,$
its Veronese subrings $C^{(s)}$ satisfy $C^{(s)}=k[C_s]$
whenever $s \geq \max\{n-2, 1\}.$ 
For our purpose, a weaker condition will
suffice, namely,
\begin{equation}\label{Birationalveroneseindex}\Quot(C^{(s)}) = \Quot(k[C_s])  \, .
\end{equation}

We call the least positive integer $s$ for which \eqref{Birationalveroneseindex} holds the \emph{birational Veronese index} of the ideal $I$ relative to the Rees valuation $v$.
In particular, this invariant is $\leq \max\{n-2,1\}$.

Let
$\mathcal C$ be the monomial basis of the vector space $C_s,$ which is  the collection 
of all monomials $\gamma \in \overline{I^s}$ with $v(\gamma)=s \cdot v(I).$ 
Let $\mathcal B$  denote the monomial basis of $B_1, $ which is the set of all monomials $\beta \in \overline{I}$ with $v(\beta)=v(I).$

}
\end{Discussion}

\begin{Proposition}\label{Lfielddegree2} With notation as in Discussions~$\ref{Dgrading}$, ~$\ref{log2}$, ~$\ref{DiscussionBiratVerInd}$
we have
\[ [(\overline{\cR}/\fQ):(\cR/\fP)] = [C:A]= s \cdot \frac{\mid \mathcal{I}_n({\rm log}(\mathcal A)) \mid}{\mid \mathcal{I}_n({\rm log}(\mathcal C)) \mid} \, . \]
In particular, 
$$\lambda(G_{\fP})= v(I)  \cdot s \cdot \frac{\mid \mathcal{I}_n({\rm log}(\mathcal A)) \mid}{\mid \mathcal{I}_n({\rm log}(\mathcal C)) \mid} \, .$$
\end{Proposition}
\begin{proof}Again, by Proposition~\ref{1-1} and Lemma~\ref{LIso} it suffices to prove the
second equality in the first formula. Also recall from Discussion~\ref{Dgrading} that
the algebras $A,$ $C,$ and their
Veronese subrings have the same dimension as $R;$ therefore, the 
field degrees in the formulas below will all be finite. Now,
\[
\begin{array}{rcll}
[C:A]
                &=& 
                [C^{(s)}:A^{(s)}] \hspace{2.1cm} \text{by Lemma~\ref{LemmaVeronese}}\\[2mm]
 &=& 
                 \dfrac{[R:A^{(s)}]}{[R:C^{(s)}]} \\[5mm]

                &=& 
                 s \cdot \dfrac{[R:A]}{[R:C^{(s)}]} \hspace{1.95cm} \text{by Lemma~\ref{LemmaVeronese}}\\[5mm]

                 &=& 
                 s \cdot \dfrac{[R:A]}{[R:k[C_s] \, ]}\hspace{1.75cm} \text{by (\ref{Birationalveroneseindex})} \, . \\[5mm]
                   
\end{array}
\]
Finally, as is well known (and can be seen by passing to the Laurent polynomial ring
and using the Smith normal form of the log matrix), we have 
 \begin{equation*}
 [R:A]=[R:k[\mathcal A] \, ]=| \mathcal{I}_n({\rm log}(\mathcal A))| \ \ \ {\text{and}} \ \ \
[R:k[C_s] \, ]=[R:k[\mathcal C] \,]= |\mathcal{I}_n({\rm log}(\mathcal C))|.\qedhere
 \end{equation*}
\end{proof}    

Notice that if the birational Veronese index  in Discussion~\ref{DiscussionBiratVerInd} is equal to 1, 
then $\mathcal C= \mathcal B.$
We finish with a variation of the proposition in this case. 

\begin{Remark} Assume that $\Quot(C) = \Quot(k[C_1])  $, that is,
the birational Veronese index is equal to $1$.
Let $N \subset H$ be the subgroups of $\mathbb Z^n$ generated by the image 
of $\, {\rm log}(\mathcal A)$ and of $\, {\rm log}(\mathcal B),$ respectively. Then  $[C:A]={\rm o}(H/N),$
the order of the group $H/N.$
\end{Remark}
\begin{proof}
It follows by Proposition~\ref{Lfielddegree2}, as  $| \mathcal{I}_n({\rm log}(\mathcal A)) |
={\rm o}(\mathbb Z^n/N) $ and  $| \mathcal{I}_n({\rm log}(\mathcal B)) |
={\rm o}(\mathbb Z^n/H).$ 
\end{proof}

We now return to the computation of Behrend functions. Thus, we will assume that $k=\mathbb C$
and $I$ is zero-dimensional. In this case, each of the minimal primes $\fP_i$ of $I \cR$
contracts to $\m$ and therefore satisfies the assumptions of Discussion~\ref{Dgrading}. Since
we are dealing with several primes $\fP_i,$ we denote the corresponding objects defined in
Discussions~\ref{Dgrading}, ~\ref{log2}, ~\ref{DiscussionBiratVerInd} by $v_i,$ $A_i, B_i, C_i$ $\mathcal P_i,$ $s_i,$ $\mathcal A_i,$ $\mathcal A'_i,$ $\mathcal A''_i,$ $\mathcal B_i,$ 
$\mathcal C_i.$ 

We  state the main result of this section.

\begin{Theorem}\label{TheoremBehrensMonomials}Let $R=\mathbb C[x_1, \ldots, x_n]$ be a polynomial ring
and $I$ be a zero-dimensional monomial ideal. Then  with notation as in Discussions~$\ref{Dgrading}$, $\ref{log2}$, ~$\ref{DiscussionBiratVerInd}$, we have
\begin{align*}
 \nu_I & 
 = \sum_{i} v_i(I) \cdot \frac{(n-1)! \ {\rm{vol}}(\mathcal P_i)}{e(A_i)} 
 \\ 
 &= \sum_{i} v_i(I)  \cdot | \mathcal{I}_n({\rm log}(\mathcal A_i') )| 
 \\
 &= \sum_{i} v_i(I)  \cdot | I_{n-1}({\rm log}(\mathcal A_i'') )|
 \\
 &=\sum_{i}  v_{i}(I)\cdot s_i \cdot \frac{\mid \mathcal{I}_n({\rm log}(\mathcal A_i)) \mid}{\mid \mathcal{I}_n({\rm log}(\mathcal C_i)) \mid} \, ,   
\end{align*}
where the sum is taken over all Rees valuations $v_i$ of $I.$
\end{Theorem}

\begin{proof}
It follows 
by combining Proposition~\ref{1-1}, Theorem~\ref{Lfielddegree}, and Proposition~\ref{Lfielddegree2}.
\end{proof}

\vspace{.07cm}

\begin{Corollary}\label{nice} 
If the 
birational Veronese indices of $I$ are equal to 1 for all Rees valuations,
then 
\[ \nu_I= \sum_{i}  v_{i}(I) \cdot \frac{\mid \mathcal{I}_n({\rm log}(\mathcal A_i) )\mid}{\mid \mathcal{I}_n({\rm log}(\mathcal B_i)) \mid} \, . \]
If, in addition, $I$ is integrally closed,  then
\vspace{-.1cm}
\[ \nu_I= \sum_i v_i(I) \, . \]
\end{Corollary}

\vspace{.13cm}

By Discussion~\ref{DiscussionBiratVerInd}, we obtain the following special case.

\begin{Corollary}\label{Corn3} Let ${\mathbb C}[x_1,\ldots,x_n]$ be a polynomial ring with $n\le 3$ and $I$  an integrally closed zero-dimensional monomial ideal. Then 
\vspace{-.2cm}
$$\nu_I=\sum v(I)\, ,$$
where the sum is taken over all Rees valuations of $\, I.$
\end{Corollary}

\begin{NotationDiscussion}\label{Dnozerodim}{\rm Although this has no bearing on the computation of Behrend functions yet, for future reference we finish this section with an algorithm for determining the field degrees 
$[\overline{\cR}/\fQ : \cR/\fP]$ when $I$ is a monomial ideal of arbitrary dimension. 
Let $R=k[x_1, \ldots, x_n]$ be a polynomial ring and  $I$  any monomial ideal.
Again, let $\fP$ be a minimal prime of $I\cR,$  $\fQ$  the unique prime of $\overline{\cR}$ lying over it, and $v$  the corresponding Rees valuation of $I.$
The prime ideals $\fP$ cannot all contract to $\m$ unless $I$ is zero-dimensional. 
In general, 
however, $\fP\cap R$ is a prime monomial ideal, hence, it is generated by a subset of  variables $\{x_{i_1}, \ldots, x_{i_r}\}$. Specializing the remaining variables to $1,$ we obtain a map $R\lto R^{\sharp}=k[x_{i_1}, \ldots, x_{i_r}]$ which, up to renaming the variables and adjoining 
Laurent polynomial variables, is a localization map.
Hence, on the one hand the ideal $I^{\sharp}:=IR^{\sharp}$ is a monomial ideal,  on the other hand
$\cR^{\sharp}:=\cR(I^{\sharp})$ and $\overline{\cR^{\sharp}}$ are obtained from $\cR $ and $\overline{\cR}$
by tensoring with $\otimes_R R^{\sharp}.$ Moreover, $\fP^{\sharp}:=\fP \cR^{\sharp}$ and $\fQ^{\sharp}:=\fQ \overline{\cR^{\sharp}}$ are minimal primes of $I^{\sharp} \cR^{\sharp}$ and of $I^{\sharp}\overline{\cR^{\sharp}},$ respectively. The field degree remains the same,  $[(\overline{\cR}/\fQ):(\cR/\fP)] = [(\overline{\cR^{\sharp}}/\fQ^{\sharp}):(\cR^{\sharp}/\fP^{\sharp})], $ but the ideal $\fP^{\sharp}$ contracts to the maximal homogeneous ideal of the ambient ring $R^{\sharp}.$ So we may apply Theorem~\ref{Lfielddegree}, and Proposition~\ref{Lfielddegree2}, with $s^{\sharp},$ $\mathcal A^{\sharp},$ $(\mathcal A^{\sharp})',$ $(\mathcal A^{\sharp})''$ and $\mathcal C^{\sharp}$ defined for $I^{\sharp}\subset R^{\sharp}$ as in Discussions ~\ref{log2} and ~\ref{DiscussionBiratVerInd}. When combined with Proposition~\ref{1-1}, this gives the following
formula.}
\end{NotationDiscussion}

\begin{Corollary} Let $R=k[x_1, \ldots, x_n]$ be a polynomial ring, $I$ be any monomial ideal, 
and $\fP$ be minimal prime of $I\cR.$ With notation as in Discussion~$\ref{Dnozerodim}$, one has
$$\lambda(G_{\fP})= v(I)  \cdot | I_r({\rm log}((\mathcal A^{\sharp})')) | = v(I)  \cdot | I_{r-1}({\rm log}((\mathcal A^{\sharp})'') )| =v(I)  \cdot s^{\sharp} \cdot \frac{\mid I_r({\rm log}(\mathcal A^{\sharp})) \mid}{\mid I_r({\rm log}(\mathcal C^{\sharp})) \mid} \, .$$
    
\end{Corollary}
\vspace{.2cm}

 From the point of view of valuation theory, the ideal $(x_{i_1}, \ldots, x_{i_r})R$ is the center of the Rees valuation $v$ of $I.$ In terms of the Newton Polyhedron,  specializing from $I$ to $I^{\sharp}$ has the effect of projecting  
the Newton Polyhedron of $I$  onto the subspace of $\mathbb R^n$ spanned by  $\{{\bf e}_{i_1}, \ldots, {\bf e}_{i_r}\}.$

\vspace{.2cm}

\section{Ideals with  one Rees valuation and weighted Veronese subrings}\label{SectionOneVal}

This section centers on a phenomenon we find quite striking: zero-dimensional monomial 
ideals possessing a single Rees valuation behave, for the purpose of 
computing the Behrend function, \emph{as if they were normal}---even 
though they seldom are. 
Recall from Section~\ref{SectionBehrendMonomial} 
that our  formula for $\nu_I$ involves 
the birational Veronese indices of $I$, 
some auxiliary positive 
integers  which in general can be 
greater than~$1$ and measure the failure of certain 
extensions $B \subset C$ to be birational. 
We prove that if a zero-dimensional monomial ideal $I$ has a 
single Rees valuation, equivalently, if $(x_1^{a_1}, \ldots, x_n^{a_n}) 
\subset I \subset \overline{(x_1^{a_1}, \ldots, x_n^{a_n})}$ 
for some positive integers $a_1, \ldots, a_n$,
the birational Veronese index of $I$ is equal to $1$ (Corollary~\ref{CorBiratVerIndexOneVal}). 
Then  the formula for the Behrend 
function of such an ideal collapses to the elegant closed formula of 
Theorem~\ref{Theoremonlyone}
$$
\nu_I = |\mathcal{I}_n(\log \mathcal{A})|,
$$  
expressed in terms of a single log 
matrix, whose columns correspond simply to the generators of $I$ of minimal value. 
In the case $I = \overline{(x_1^{a_1}, \ldots, 
x_n^{a_n})},$ this further simplifies to the formula of Corollary~\ref{Onlyone}
$$
\nu_I = 
\mathrm{lcm}(a_1, \ldots, a_n),
$$  generalizing 
\cite{Graffeo}*{Theorem~B}  to an arbitrary number of variables.

The proof of Corollary~\ref{CorBiratVerIndexOneVal} rests on
a result in polyhedral combinatorics of independent interest,  which we now discuss.
A lattice polytope $\mathcal{P}\subset \bR^n$ satisfies the \emph{integer decomposition property} (IDP) if every lattice point of every dilation $d\cP$ is a sum of $d$ lattice points of $\cP$,
equivalently, if the polytopal ring and the Ehrhart ring of $\cP$ coincide.
This property is well studied due to its several notable consequences.
In algebraic terms, it implies the normality of the polytopal ring and the generation of the Ehrhart ring  in degree 1.
In terms of the associated toric variety, it implies that $\cP$ defines a very ample line bundle and the corresponding embedding is projectively normal.
The IDP also has important implications  for the structure of $h^*$-vectors
(see, for instance, \cite{APPS})
and in
integer programming (see, for instance, \cite{BaumTrotter,Schrijver}).
For a collection of results on this problem, we refer the reader to 
 \cite{projectivenormality}.

In this section, we are interested in the following  class of polytopes.

\begin{Notation}\label{NotationWeightedStandard}{\rm 
Let $\ell_1, \ldots, \ell_n \in \bZ_{>0}$ be positive integers,
$M$  a positive common multiple of these integers, and let
\begin{align*}
\Delta &=\Delta(\ell_1, \ldots, \ell_n;M) = \left\{ \mathbf{u} \in {\bR}_{\geq 0}^n \,:\, \sum_{j=1}^n \ell_j u_j = M\right\},\\
\mathcal{F} &=\mathcal{F}(\ell_1, \ldots, \ell_n;M) = \left\{ \mathbf{u} \in {\bZ}_{\geq 0}^n \,:\, \sum_{j=1}^n \ell_j u_j = M\right\} = \Delta\cap \bZ^n,\\
\mathcal{S} &=\mathcal{S}(\ell_1, \ldots, \ell_n;M) = \left\{ \mathbf{u} \in {\bZ}_{\geq 0}^n \,:\, \sum_{j=1}^n \ell_j u_j \equiv 0 \ \ {\rm mod} \, M\right\}
= \bigcup_{d \in \bZ_{\geq 0}}d\Delta\cap \bZ^n.
\end{align*}
The lattice polytope $\Delta$ is  called a \emph{weighted standard simplex}.
The set $\cF$ consists of the lattice points of $\Delta$,
while $\cS$ is the monoid consisting of all the lattice points in the conical hull of $\Delta$.

Let $R = k[x_1, \ldots, x_n]$ be a polynomial ring with positive grading $\deg(x_i) = \ell_i$ for $i = 1, \ldots, n$.
Following Notation and Discussion \ref{log2}, 
let $C \subset R$ denote the Ehrhart ring of $\Delta$, 
that is, the monomial subalgebra of $R$ whose monomials corresponds to the lattice points of $\cS$,
and let $B \subset C$ denote the polytopal ring  of $\Delta$,
that is, the subalgebra generated by monomials corresponding to the lattice points of $\cF$.
Observe that $C= R^{(M)}$ is the $M$-th Veronese subring of $R$,
while $B = k[R_M]$ is the subalgebra generated in degree $M$.
In particular, $R^{(M)}$ is a finite module over $k[R_M]$.
By scaling the degree, the latter can be considered as a standard graded $k$-algebra, however, 
in this section it is convenient to use the positive grading inherited by $R$.
}
\end{Notation}

We point out that the projective variety $ \bP(\ell_1, \ldots, \ell_n):= \Proj R $ is known as a \emph{weighted projective space}, a  class of toric varieties that has received substantial interest recently, 
and the polytope $\Delta(\ell_1, \ldots, \ell_n;M)$ defines the line bundle $\cO(M)$ on 
$\bP(\ell_1, \ldots, \ell_n)$
(see, for instance,  \cite{ABL,BanksRamkumar,MullerPaemurru}).

By Notation \ref{NotationWeightedStandard},
it follows that a weighted standard simplex $\Delta$ satisfies IDP if and only if $B=C$, equivalently, if and only if  the monoid $\cS$ is generated by $\cF$.
Even for weighted standard simplices, 
the IDP fails in general;
a counterexample is 
$(\ell_1, \ldots, \ell_4) = (1, 6, 10, 15)$
\cite{ogata2005k};
see \cite{BanksRamkumar,BDHLS,SHZ}
for some very recent work on the IDP for simplices.

In order to better understand this failure, 
consider the diagram 
\[
\begin{tikzcd}
B = k[R_M] \arrow[r, hook] \arrow[d, hook] & C = R^{(M)} \arrow[d, hook] \\
\operatorname{Quot}(B) \arrow[r, hook]    & \operatorname{Quot}(C)
\end{tikzcd}
\]
and recall that
$C$ is a normal domain and $B \subset C $ is an  integral extension. 
It follows that
\vspace{1mm}
\begin{center}
IDP holds $\Leftrightarrow B$ is normal and     $B \subset C$ is birational.\footnote{
Note that several literature sources erroneously state that IDP is equivalent to normality.
}
\end{center}
\vspace{1mm}
While normality of $B$ fails in general, the main result of this section is that 
the extension $B\subset C$ is always birational for 
weighted standard simplices (Corollary \ref{CorBirationalVeroneseSubring}).
What makes this result even more remarkable is the fact that, 
in general,
the IDP may fail for the opposite reason:
 Example~\ref{ExFails} exhibits a polytope -- in fact, a simplex -- 
 for which $B$ is normal but $B \subset C$ is not birational.
Nevertheless, 
in Example~\ref{E5.10} we compute the Behrend function for a class of
zero-dimensional ideals arising from this example, confirming that the 
machinery of Theorem~\ref{TheoremBehrensMonomials} handles cases beyond 
the reach of Theorem~\ref{Theoremonlyone}.

\begin{Lemma}\label{LemmaGCD}
We use the setting of $\ref{NotationWeightedStandard}$.
Assume that $n \geq 2$ and $\gcd(\ell_1, \ldots, \ell_n)= 1$.
For each $i = 1, \ldots, n$, define
\begin{align*}
\chi_i &= \gcd \left\{ u_i \,:\, \mathbf{u} = (u_1, \ldots, u_i, \ldots, u_n) \in \mathcal{F} \right\},\\
\gamma_i &=  \gcd(\ell_1, \ldots, \widehat{\ell_i}, \ldots, \ell_n).
\end{align*}
Then  we have $\chi_i = \gamma_i$ for each  $i$.
\end{Lemma}

\begin{proof}
For simplicity, without loss of generality, we assume $i=1$, and we need to show that $\chi_1 = \gamma_1$.
It is easy to see that $\gamma_1$ divides $\chi_1$.
Indeed, 
let $\mathbf{u}\in \mathcal{F}$, so that $ \ell_1 u_1= M - \ell_2u_2 - \cdots -  \ell_n u_n$.
Since $\gamma_1$ divides the right-hand side
and $1=\gcd(\ell_1, \ldots, \ell_n) = \gcd(\gamma_1, \ell_1)$, it follows that $\gamma_1$ divides $u_1$.

We  prove that $\chi_1$ divides $\gamma_1$,
 by induction on the length of the prime factorization of $\ell_1$.

For the base of the induction,  assume $\ell_1=1$.
For each $j>1$, consider the relation 
$$
\ell_j \cdot 1 + \left( \frac{M}{\ell_j}-1\right)\cdot \ell_j = M.
$$
Notice that the coefficients are nonnegative integers since $M$ is divisible by $\ell_j.$
Denoting by $\mathbf{e}_h\in \bZ^n$ the standard basis vectors for $h = 1, \ldots, n$,  the relation implies that $\ell_j \mathbf{e}_1+ \left( \frac{M}{\ell_j}-1\right) \mathbf{e}_j \in \mathcal{F}$. Since $\gcd(\ell_2, \ldots\,\ell_n)=\gamma_1$, it follows that $\chi_1 $ divides $\gamma_1$.

For the  induction step, let $p$ be a prime dividing $\ell_1$. 
For each $j$, define
$$
\ell'_j = \begin{cases}
    \ell_j & \text{ if } \gcd(\ell_j,p)=1\,,\\
    \frac{\ell_j}{p} & \text{ if } p\,|\,\ell_j\,,\\
\end{cases}
\qquad
M' = \frac{M}{p}.
$$
In particular, $\ell'_1 = \frac{\ell_1}{p}$.
Denote 
$\mathcal{F}' =\mathcal{F}(\ell'_1, \ldots, \ell'_n;M')$, and define $\chi_j', \gamma'_j$ accordingly.
By induction, we have $\chi'_1 = \gamma'_1$. 
Now, we  relate them to $\chi_1, \gamma_1$.
Let $\mathbf{y} = (y_1, \ldots, y_n) \in \mathcal{F}'$, that is, 
$$
\ell'_1 y_1+ \ldots  +  \ell'_n y_n = M'.
$$
We  multiply this equation by $p$, and obtain a relation among the $\ell_j$'s adding up to $M$.
Specifically, if $\gcd(\ell_j,p)=1$, then $\ell_j = \ell'_j$, 
and we absorb  $p$ in the coefficient $py_j$, otherwise
 we absorb  it   in $\ell_j=p\ell'_j$.
Thus, letting
$$
z_j = \begin{cases}
    py_j & \text{ if } \gcd(\ell_j,p)=1\,,\\
    y_j & \text{ if } p\,|\,\ell_j\,,\\
\end{cases}
$$
we have  $\mathbf{z} = (z_1, \ldots, z_n) \in \mathcal{F}$ and
 $z_1=y_1$.
By considering such $\mathbf{z} \in \mathcal{F}$ for all $\mathbf{y} \in \mathcal{F}'$,
we conclude that $\chi_1$ divides $\chi'_1$.
On the other hand, $\gamma'_1$ clearly divides $\gamma_1,$  so $\chi_1$ divides $\gamma_1,$
as required.
\end{proof}

\begin{Theorem}\label{ThmSameSubgroups}
The sets  $\cF$ and $\cS$ generate the same subgroups of $\, \bZ^n$.
\end{Theorem}
\begin{proof}
We begin by observing that the subgroup  
$\mathrm{gp}(\cS) \subseteq \bZ^n$ generated by $\cS$ is equal to 
\begin{equation}\label{EqSubgroupGenerated}
     \left\{ \mathbf{v} \in \bZ^n \,:\, \sum_{j=1}^n \ell_j v_j \equiv 0 \ \ {\rm mod} \, M \right\}.
\end{equation}
The inclusion of 
$\mathrm{gp}(\cS)$ in \eqref{EqSubgroupGenerated}
is clear.
For the reverse inclusion,
observe that 
$\cS$ contains a positive 
integer multiple of every standard basis vector $\mathbf{e}_j \in \bZ^d$,
since each $\ell_j$ divides $M$.
Thus, for any 
$
\mathbf{v} \in \bZ^n$ such that $\sum_{j=1}^n  \ell_j v_j\equiv 0 \ \ {\rm mod} \, M, 
$
it is possible to find  $\mathbf{u}\in \cS $ such that 
$
\mathbf{u}'=
\mathbf{v}+\mathbf{u} \in \bZ_{\geq 0}^n$,
therefore $\mathbf{u}'\in \cS $ too,
 and so $ 
\mathbf{v}=\mathbf{u}'-\mathbf{u}  \in  \mathrm{gp}(\cS)$.

In order to prove the theorem, we must show that every $\mathbf{v}\in
\mathrm{gp}(\cS)$ is equal to a linear combination of elements of 
$\mathcal{F}$ with  integer coefficients.

We may assume without loss of generality that $\, \gcd(\ell_1, \ldots, \ell_n)= 1,$ since the sets $\cS$ and
$\mathcal{F}$ do not change if we divide  $\ell_1, \ldots, \ell_n$ and $M$ 
by $\gcd(\ell_1, \ldots, \ell_n)$.

We prove the statement
by induction on $n$. 
The base case  $n = 1$ is trivial, since $\, \cS= M\, \mathbb Z_{\geq 0} \, $ and $\mathcal{F}=\{M\},$ so  we assume $n > 1.$
Let $\mathbf{v}\in
\mathrm{gp}(\cS)$,
that is, 
$\mathbf{v}=(v_1, \ldots,v_n) \in \bZ^n$
such that $\, \sum_{j=1}^n  \ell_j v_j\equiv 0 \ {\rm mod }\, M$.
It follows from Lemma \ref{LemmaGCD} and Bézout's identity that
there exists $\mathbf{z}=(z_1, \ldots, z_n)$ in the subgroup $ \langle \mathcal{F}\rangle_\bZ$ such that $$z_n = \chi_n=\gamma_n = \gcd(\ell_1,  \ldots, \ell_{n-1})\, .$$
Since $1=\gcd(\ell_1, \ldots, \ell_n) = \gcd(\gamma_n, \ell_n)$, it follows from $\sum_{j=1}^n  \ell_j v_j \equiv 0 \ {\rm mod}\, M\,$ that $\gamma_n$ divides $v_n$.
Let $b = \frac{v_n}{\gamma_n}\in \bZ$,
then $\mathbf{w}=\mathbf{v}-b\mathbf{z}\in \mathrm{gp}(\cS)$ and it 
has $n$-th coordinate equal to 0.
In order to conclude the proof, it suffices to show that $\mathbf{w}$ is a $\bZ$-linear combination of elements of $\mathcal{F}.$

By abuse of notation, we identify vectors of $\bZ^n$ whose $n$-th coordinate is zero with vectors of $\bZ^{n-1}.$
Since $\ell_1 w_1+ \ldots + \ell_{n-1} w_{n-1} \equiv 0 \ {\rm mod}\, M\,$
and $L' = \mathrm{lcm} (\ell_1, \ldots,  \ell_{n-1})$ divides $L$,
we have $\mathbf{w} \in\langle \cS' \rangle_\bZ$, where $\cS' = \cS(\ell_1,   \ldots, \ell_{n-1};M)\subset \bZ^{n-1}$.
Similarly, let
$\mathcal{F}' = \mathcal{F}(\ell_1, \ldots, \ell_{n-1};M)\subset \bZ^{n-1}$.
By induction, $\mathbf{w}$ is equal to a $\bZ$-linear combination of elements of $\mathcal{F}'.$
But we have $\mathcal{F}'\subset \mathcal{F}$ under our identification, so the proof is complete.
\end{proof}

We now give two applications of Theorem~\ref{ThmSameSubgroups}, which are in fact 
equivalent to the statement of the theorem.

\begin{Corollary}\label{CorBirationalVeroneseSubring}
Let $R=k[x_1, \ldots, x_n]\, $ be a  polynomial ring with $\deg(x_i)=\ell_i \in \bZ_{>0}$.
Let $M$ be a positive multiple of $\ell_1, \ldots, \ell_n$ and $R^{(M)}=\bigoplus_{j\geq 0}R_{jM}$
the $M$-th Veronese subring.
Then the extension $k[R_M] \subset R^{(M)}$ is birational, equivalently, $e(k[R_M])=e(R^{(M)}).$
\end{Corollary}
\begin{proof}The birationality of the extension is an immediate consequence of Theorem~\ref{ThmSameSubgroups}. The claim about multiplicities follows because the extension
$k[R_M] \subset R^{(M)}$ is module finite and, 
after rescaling grading, the ring $k[R_M]$ becomes a standard graded $k$-algebra. 
\end{proof}

\begin{Corollary}\label{CorBiratVerIndexOneVal}
   Let $R=k[x_1, \ldots, x_n]$ be a polynomial
ring, $a_1, \ldots, a_n$ be positive integers, and $I$ be a monomial ideal with
$(x_1^{a_1}, \ldots, x_n^{a_n}) \subset I \subset \overline{(x_1^{a_1}, \ldots, x_n^{a_n})}.$
Then the  birational Veronese index of $I$ with respect to its only Rees valuation is equal to 1.
\end{Corollary}

\begin{proof}
Let $L = \lcm(a_1, \ldots, a_n)$ and $\ell_i = \frac{L}{a_i}$ for $i = 1, \ldots, n.$
Setting $\deg(x_i) = \ell_i$, 
with notation as in Discussions \ref{Dgrading},
one has
$k[R_L]=B$ and $R^{(L)}=C.$
The conclusion follows from Corollary~\ref{CorBirationalVeroneseSubring} and Discussion \ref{DiscussionBiratVerInd}.
\end{proof}

Next, we show that the main result of this section is rather sharp:
it may fail for lattice simplices whose vertices are not on the coordinate axes,
even when the simplex is a normal polytope.

\begin{Example}\label{ExFails}{\rm 
Consider the ideal $I= (x_1x_2,x_1x_3,x_2x_3,x_4x_5,x_4x_6,x_5x_6) \subset R = k[x_1,\ldots, x_6]$.
This is a well-known example of a squarefree monomial ideal that is integrally closed but not normal: the monomial $m = x_1x_2x_3x_4x_5x_6$ is not in $I^3$ but it is in $\overline{I^3}$ since 
$$m^2 = (x_1x_2)(x_1x_3)(x_2x_3)(x_4x_5)(x_4x_6)(x_5x_6) \in I^6=(I^3)^2.$$
In fact, even the weaker birationality property of Corollary~\ref{CorBiratVerIndexOneVal} fails, as we are going to show now.

The Newton polyhedron of $I$ has only one compact face, which is precisely the simplex whose vertices correspond to the minimal monomial generators of $I$. 
In terms of valuations, 
it follows that $I$ has only one Rees valuation  $v$ centered at the maximal homogeneous ideal; since $I$ is equigenerated, this is the valuation given by the standard grading. 
We have $v(I)=2$, $m \in \overline{I^3}$, and  $v(m) = 6 =3 v(I)$. 
In terms of the rings $A,B,C$ from Discussion \ref{Dgrading}, we have $A=B$ and $m \in C_3$.
But $m$ is not in the quotient field of $B$: this can be seen by introducing a $\bZ^2$-grading on $R$ where $\deg(x_i) = (1,0)$ for $i =1,2,3$ and 
$\deg(x_i) = (0,1)$ for $i =4,5,6$.
The graded $k$-algebra $B$ is generated by $B_1,$ and $B_1$
 is concentrated in degrees $(2,0)$ and $(0,2).$ On the other hand, 
 the homogeneous element $m$ has degree $(3,3).$ If $m$ were in ${\rm Quot}(B),$ then $m$ would be the quotient of nonzero homogeneous elements of $B,$ which is impossible for degree reasons.
 This shows that the extension $B \subset C$ is not birational.
 
 Finally, we observe that the monomials $x_1x_2,x_1x_3,x_2x_3,x_4x_5,x_4x_6,x_5x_6$ are algebraically independent over $k$, hence, 
the polytopal ring $B$ is isomorphic to a polynomial ring, in particular, it is a normal domain. 
 }
 \end{Example}

While the previous example is not a zero-dimensional ideal, it is straightforward to turn it into a class of zero-dimensional examples. In  Example~\ref{E5.10}, we   compute the Behrend function for this class. 
For this, we need the following lemma.

\begin{Lemma}\label{transversal}
Let $
R = k[x_1,\dots,x_p,y_1,\dots,y_q]$
 be a polynomial ring,
 $I \subset k[x_1,\dots,x_p]$ and $J \subset k[y_1,\dots,y_q]$ be monomial ideals, and set
\[
K = IR + JR \subset R.
\]
Let $v_I$ and $v_J$ be Rees valuations of $I$ and $J$.
Write
$
v_I(I)=d_I, v_J(J)=d_J,
$
and  $D=\mathrm{lcm}(d_I,d_J)$. Then the monomial valuation $v$ on ${\rm Quot} (R)$ satisfying
\[
v(x^a y^b)
=
\frac{D}{d_I} v_I(x^a)
+
\frac{D}{d_J} v_J(y^b)
\qquad \text{for} \quad a\in\mathbb Z_{\ge 0}^p\, ,\ b\in\mathbb Z_{\ge 0}^q
\]
is a Rees valuation of $K,$ and $v(K)=D.$ Moreover, every Rees valuation of $K$ arises 
in this way from a unique pair $(v_I,v_J)$.
\end{Lemma}

\begin{proof} Rees valuations of monomial ideals are in bijection with the non-coordinate 
facets of their Newton polyhedra. Thus, we may harmlessly assume that the field $k$ is algebraically closed.
Let
\[
{\bf u}\cdot\alpha = 1\,, \qquad {\bf w}\cdot\beta = 1
\]
be the equations defining the facets $\mathcal F$ and $\mathcal G$ corresponding  to $v_I$ and $v_J$, respectively, where
\[
{\bf u}=(u_1, \ldots, u_p)\in\mathbb Q_{\ge0}^p\,,\qquad {\bf w}=(w_1, \ldots, w_q)\in\mathbb Q_{\ge0}^q\,,\qquad \alpha\in
\mathrm{NP}(I)\,,\qquad \beta\in\mathrm{NP}(J)\,,
\]
and $\cdot$ denotes the standard dot product.
The value $d_I$ (resp.\ $d_J$) is the smallest positive integer such that $d_I{\bf u}$ (resp. $d_J{\bf w}$) has integral coordinates, and then $v_I(x_i)=d_Iu_i$ (resp.  $v_J(y_j)=d_Jw_j$).

The Newton polyhedron of $K=I+J$ 
is given by
\begin{equation}\label{EqNewton}\mathrm{NP}(K) = \ \mathrm{conv}\left( (\mathrm{NP}(I)\times  \{ {\bf 0}\} )\cup (\{ {\bf 0}\}
\times\mathrm{NP}(J)) \right) + \mathbb R_{\ge 0}^{p+q} \, .
\end{equation}

We first claim that the equation
\[
{\bf u}\cdot\alpha + {\bf w}\cdot\beta = 1
\]
defines a non-coordinate facet of $\mathrm{NP}(K).$ Indeed, the
inequality 
\[{\bf u}\cdot\alpha + {\bf w}\cdot\beta \geq 1\] holds if $(\alpha, \beta)$
belongs to ${\rm NP}(I) \times \{ {\bf 0}\}$ or to $\{ {\bf 0}\} \times {\rm NP}(J)$. 
Thus, the same inequality
holds for every $(\alpha, \beta)$ in the
convex hull of $(\mathrm{NP}(I)\times  \{ {\bf 0}\} ) \cup ( \{ {\bf 0}\}
\times\mathrm{NP}(J)),$ and hence for every point in ${\rm NP}(K).$ 
Likewise, the equation \[
{\bf u}\cdot\alpha + {\bf w}\cdot\beta = 1
\] holds for every point in $\mathcal F \times \{ {\bf 0}\}$ and every point
in $\{ {\bf 0}\} \times \mathcal G$, hence for every point in the convex hull of
$(\mathcal F \times \{ {\bf 0}\}) \cup (\{ {\bf 0}\} \times \mathcal G),$ which has full dimension,
$p+q-1.$ Thus, the set $\{(\alpha, \beta) \in {\rm NP}(K) \mid u\cdot\alpha + w\cdot\beta = 1\}$
is indeed a facet of ${\rm NP}(K),$  necessarily a non-coordinate facet.

The equation with minimal nonnegative integer coefficients defining this facet is  
\[
D {\bf u} \cdot \alpha + D {\bf w} \cdot \beta = D \, .
\]
Hence, the Rees valuation $v_K$ corresponding to this facet satisfies $v_K(K)=D$ and
hence $v_K(x_i)=\frac{D}{d_I}v_I(x_i)$ and $v_K(y_j)=\frac{D}{d_J}v_I(y_j),$ as claimed.

We now prove that every Rees valuation of $K$ arises in this way from a unique pair $(v_I,v_J).$ In light of the first
part of the proof, 
it suffices to show that the number of Rees valuations of $K$
is at most the number of Rees valuations of $I$ times the number of 
Rees valuations of $J.$ To this end, we write $S=k[x_1, \ldots ,x_p],$ $T=k[y_1, \ldots, y_q],$ $G_1={\rm gr}_I(S),$ $G_2={\rm gr}_J(T),$ and $G={\rm gr}_K(R).$ Owing to Proposition~\ref{1-1}, we only need to prove that the number of minimal primes
of $G$ is at most the number of minimal primes of $G_1$ times the number of 
minimal primes of $G_2.$ 

The natural maps $G_1 \rightarrow G$ and $G_2 \rightarrow G$ induce a homomorphism of 
$k$-algebras $G_1\otimes_kG_2 \rightarrow G,$ which is surjective. The associated graded ring
$G$ is equidimensional of dimension $p+q,$ and $G_1\otimes_kG_2$ has dimension 
at most $p+q,$ as can be seen, for instance, by considering Noether normalizations. 
Thus, every minimal prime of $G$ contracts to a minimal prime of 
$G_1\otimes_kG_2.$ 
It remains to show that the number of minimal primes of
$G_1\otimes_kG_2$ is at most the number of minimal primes of $G_1$ times the
number of minimal primes of $G_2.$

Let $\mathfrak{P}$ be a minimal prime of $G_1\otimes_kG_2.$ The map
$G_1 \rightarrow G$ is flat, hence satisfies going down, which implies that 
$\fP$ contracts to a minimal prime $\fP_1$ of $G_1.$ Likewise, the contraction 
of $\fP$ to $G_2$ is a minimal prime $\fP_2$ of $G_2.$ As $G_1/\fP_1$ and
$G_2/\fP_2$ are affine domains over an algebraically closed field, their
tensor product $(G_1/\fP_1)\otimes_k(G_2/\fP_2)$ is again a domain. 
Therefore,
$\fP_1G+\fP_2G$ is a prime ideal of $G.$ Since this prime ideal is contained 
in the minimal prime $\fP,$
we see that $\, \fP=\fP_1G+\fP_2G,$ as required.
\end{proof}

\begin{Example}\label{E5.10}{\rm Let $m\ge 3$ be an integer and consider the zero-dimensional 
monomial ideal $$K= (x_1^m,x_2^m,x_3^m,x_4^m,x_5^m, x_6^m, x_1x_2,x_1x_3,x_2x_3,x_4x_5,x_4x_6,x_5x_6) \subset R = \bC[x_1,\ldots, x_6]\,.$$
The birational Veronese index of $I$   is not $1$ for one of the Rees valuations, namely the valuation that gives value 1 to each variable, as explained in Example~\ref{ExFails}. 

We compute the Behrend function of $K$ by using the formula 
of Theorem~\ref{TheoremBehrensMonomials}
$$\nu_K= \sum v_i(K)  \cdot | I_{n-1}({\rm log}(\mathcal A_i''))|.$$ 

Notice that $K$ is the transversal sum of two copies of the same ideal 
$$I=(x_1^m,x_2^m,x_3^m,x_1x_2,x_1x_3,x_2x_3)\subset k[x_1,x_2,x_3]\,.$$ It is easy to see that $I$ has four Rees valuations $v_1,v_2,v_3,v_4$:
\[
\begin{array}{c|cccc}
 & v_1 & v_2 & v_3 & v_4 \\ \hline
v_i(x_1) & 1 & 1 & m-1 & m-1 \\
v_i(x_2) & 1 & m-1 & 1 & m-1 \\
v_i(x_3) & 1 & m-1 & m-1 & 1 \\ \hline
v_i(I) & 2 & m & m & m 
\end{array}
\]
\noindent
Using Lemma~\ref{transversal}, we see that the ideal $K$ has 16 valuations $v_{i,j}$ obtained 
from the pairs $(v_i,v_j).$
We list the values $v_{i,j}(K).$
\[
\begin{array}{c|cccc}
 & v_1 & v_2 & v_3 & v_4 \\ \hline
v_1 & 2 &{\rm lcm}(2,m)& {\rm lcm}(2,m) & {\rm lcm}(2,m) \\
v_2 & {\rm lcm}(2,m) & m & m & m \\
v_3 &{\rm lcm}(2,m)& m & m & m \\ 
v_4 &{\rm lcm}(2,m)& m & m & m \\ 
\end{array}
\]

\vspace{.15cm}

\noindent
In the next table, we list the integers $|I_{5}({\rm log}(\mathcal A_{i,j}''))|$ for the 16 valuations $v_{i,j};$ from the monomial generators of $K$ and the values $v_{i,j}(K)$, one readily 
obtains the log matrices ${\rm log}(\mathcal A_{i,j}''),$ and these matrices are sparse, which makes the computation easy.

 \vspace{.2cm}
 
 \[
\begin{array}{c|cccc}
 & v_1 & v_2 & v_3 & v_4 \\ \hline
v_1 & 2 &{\rm gcd}(2,m)& {\rm gcd}(2,m) & {\rm gcd}(2,m) \\
v_2 & {\rm gcd}(2,m) & m & m & m \\
v_3 &{\rm gcd}(2,m)& m & m & m \\ 
v_4 &{\rm gcd}(2,m)& m & m & m \\ 
\end{array}
\vspace{.2cm}
\]
 
\vspace{.1cm}
\noindent 
Now, Theorem~\ref{TheoremBehrensMonomials} gives
$$\nu_I=9m^2+12m+4\, .
$$
}
\end{Example}
\vspace{.4cm}

In the next theorem, we are going to determine the invariants appearing
in Section \ref{SectionBehrendMonomial} for
any monomial ideal with a single Rees valuation. 

\begin{Theorem}\label{Theoremonlyone} Let $R=k[x_1, \ldots, x_n]$ be a polynomial
ring, $a_1, \ldots, a_n$ be positive integers, and $I$ be a monomial ideal with
$(x_1^{a_1}, \ldots, x_n^{a_n}) \subset I \subset \overline{(x_1^{a_1}, \ldots, x_n^{a_n})}.$
With notation as in Discussions~$\ref{Dgrading}$, ~$\ref{log2}$, ~$\ref{DiscussionBiratVerInd}$ one has:
\begin{enumerate}[{$(a)$}]
\item $e(B)=e(C)=(n-1)! \, {\rm{vol}}(\mathcal P)= \frac{\prod a_j}{{\rm{lcm}}(a_1,\ldots,a_n)}$
\item $e(A)= \frac{\prod a_j}{|\mathcal{I}_n(\log(\mathcal{A}))|}$
\vspace{.2cm}
\item $|\mathcal{I}_n(\log(\mathcal{B}))|= {\rm{lcm}}$$(a_1, \ldots, a_n)$
\vspace{.1cm}
\item $\nu_I=|\mathcal{I}_n(\log(\mathcal{A})| \, $ when $k=\mathbb C \,.$
\end{enumerate}
\end{Theorem}
\begin{proof}By Corollary~\ref{CorBiratVerIndexOneVal}, 
the birational Veronese index of $I$ with respect to its only Rees valuation is 1.
It follows 
that $\mathcal{B}=\mathcal{C}$ and $e(B)=e(C).$ The equality
$e(C)=(n-1)! \, {\rm{vol}}(\mathcal{P})$ always holds. This proves the first
two equalities in part (a).

The ideal $I$ has a single Rees valuation $v$, and $v(I)={\rm{lcm}}(a_1, \ldots, a_n).$
Hence, combining Proposition~\ref{PropLengthHomogeneousDomain} and Theorem~\ref{Lfielddegree}(c), one sees that 
\[ \frac{e_I(R)}{e(A)}= {\rm{lcm}}(a_1, \ldots, a_n) \cdot \frac{(n-1)! \, {\rm{vol}}(\mathcal P)}{e(A)} \, . \]
In turn, since $I$ is integral over the ideal $(x_1^{a_1}, \ldots, x_n^{a_n}),$
one has 
\[ e_I(R)=e_{(x_1^{a_1},\ldots, x_n^{a_n})}(R)=\lambda(R/(x_1^{a_1}, \ldots, x_n^{a_n}))= \prod a_j \, .\]
The third equality in (a) now follows from the two displayed equations.

Combining Proposition~\ref{PropLengthHomogeneousDomain} and Proposition~\ref{Lfielddegree2}
instead, one obtains
\begin{equation}\label{Vol}
 \frac{e_I(R)}{e(A)}= {\rm{lcm}}(a_1, \ldots, a_n) \cdot \frac{|\mathcal{I}_n({\rm{log}}(\mathcal{A}))|} {|\mathcal{I}_n({\rm{log}}(\mathcal{B}))|} \, ,
 \end{equation}
which is also equal to $\nu_I$ when $k= \mathbb C.$
The set $\mathcal{B}$ only depends on the integral closure $\overline{I},$ hence
does not change if we assume, temporarily, that $I= (x_1^{a_1}, \ldots, x_n^{a_n}).$
In this case $e(A)=1$ and $|\mathcal{I}_n({\rm{log}}(\mathcal{A}))|= \prod a_j =e_I(R)$. Now
(c) follows from (\ref{Vol}).

Returning to the general case of a monomial ideal $I$ with a single Rees valuation and
substituting the value for $|\mathcal{I}_n({\rm{log}}(\mathcal{B}))|$ just found in (\ref{Vol}), one obtains (b) and (d).
\end{proof}

The next result  is a direct generalization of \cite{Graffeo}*{Theorem B}.

\begin{Corollary}\label{Onlyone}
    Let $R=\bC[x_1,\ldots, x_n]$ be a polynomial ring, let $a_1, \ldots, a_n$ be positive integers, and consider the
    ideal $I=\overline{(x_1^{a_1}, \ldots, x_n^{a_n})}.$ Then 
    \vspace{-.2cm}
    $$\nu_I={\lcm}(a_1, \ldots, a_n)\, .
    $$
\end{Corollary}
\begin{proof}This follows from parts (d) and (c) of Theorem~\ref{Theoremonlyone} because $\mathcal A= \mathcal B.$
\end{proof}

We conclude with an explicit application of the formula of Theorem~\ref{Theoremonlyone}(d).

\begin{Corollary}
    Let $R=k[x_1, \ldots, x_n]$ be a polynomial
ring, $a_1, \ldots, a_n$ be positive integers, and $I$ be a monomial ideal with
$(x_1^{a_1}, \ldots, x_n^{a_n}) \subset I \subset \overline{(x_1^{a_1}, \ldots, x_n^{a_n})}.$
Let $d={\lcm}(a_1, \ldots, a_n)$ and let $q$ be the number of elements of the set $\mathcal A$ in  Discussions ~$\ref{log2}$. Notice that $q\ge n.$ Write $\Pi=\prod a_j.$
\begin{enumerate}[{$(a)$}]    
\item If $q=n,$ so that $\mathcal A=\{x_1^{a_1},  \ldots, x_n^{a_n}\},$ then $\nu_I=\Pi\, .$
\item If $q=n+1,$ so that $\mathcal A=\{x_1^{a_1},  \ldots, x_n^{a_n}, \prod x_j^{b_j}\},$ then $\nu_I={\gcd} (\Pi,\frac{\Pi}{a_1}b_1, \ldots,\frac{\Pi}{a_n}b_n)\, .$ 
\end{enumerate}

\end{Corollary}

\vspace{.2cm}

\section{The minimal primes of the special fiber ring of a monomial ideal}
\label{SectionSpecialFiber}

In this section, we use ideas from Sections \ref{SectionBehrend} and \ref{SectionBehrendMonomial} to investigate  the components of the special fiber ring and the structure of the reduced 
fiber ring of a monomial ideal. 
We generalize, in a sense, work of Singla 
\cite{Singla}, who treated the case of monomial ideals that do not have a 
proper monomial reduction and described the reduced fiber ring and its
minimal primes in terms of maximal compact faces of the Newton polyhedron. 
We consider monomial ideals in general, using the maximal dimensional faces of the Newton polyhedron
instead. Thus, we reduce the problem to studying the special fiber rings
of simpler monomial subideals that are determined by the maximal dimensional faces.

\begin{Notation}\label{settingSection7}{\rm Throughout this section, $R=k[x_1, \dots, x_n]$ will be a polynomial ring with maximal homogeneous ideal  $\m$ and $I\subset R$ will be  a monomial ideal.
We do not assume that $I$ is zero-dimensional.
Let $v_1,\ldots,v_s$ be the Rees valuations of $I$ and set $d_i=v_i(I).$ For each $i$ we consider the monomial ideals
$$
I_i = \left( \alpha \in I  \text{ monomial } |  \, v_i(\alpha) = d_i \right).
$$

Let $\mathcal{G}=\{ \alpha_j\}$ be the  set of minimal monomial generators of $I.$ 
Write $\mathcal{G}_i=\{\alpha_j \in \mathcal{G} \ | \ v_i(\alpha_j)=d_i\}$ and notice that $G_i$ is the  set of minimal monomial generators of $I_i.$ 
Consider also the set of variables $Y=\{y_j \ | \ \alpha_j\in \mathcal{G}\} $ and  subsets $Y_i=\{ y_j \ | \ \alpha_j \in \mathcal{G}_i\}.$ 

The inclusions $I_i \subset I$ induce homomorphisms $\varphi_i: \cF(I_i) \lto \cF(I).$
We also set $
J = \sum_{i=1}^s I_i.
$

}
\end{Notation}

\begin{Proposition}\label{Prop7.2}
In the setting of Notation ~$\ref{settingSection7}$, we have 
\[
\overline{I} = \overline{J} = \sum_{i=1}^s \overline{I_i} \, .
\]
In particular, $I$ is integral over $J.$
\end{Proposition}
\begin{proof}
Let $H_i$ be the defining hyperplane of the Newton polyhedron ${\rm NP}(I)$ that corresponds to the Rees valuation $v_i.$ 
Let  $E$ and $E_i$ be the set of exponent vectors of the monomials in $\mathcal{G}$ and $\mathcal{G}_i,$ respectively, and notice that $E_i=E \cap H_i.$ 
It suffices to prove that any monomial $\alpha \in I$ is integral over $I_i$ for some $i,$ equivalently that the exponent vector $a$ of $\alpha$ is in ${\rm NP}(I_i)$ for some $i.$ 
There exists a smallest real number $\lambda$ with $0\le \lambda \le 1$ so that $\lambda a \in {\rm NP}(I),$ equivalently, $\lambda a$ is on or above each of the hyperplanes $H_i. $ 
The minimality of $\lambda$ implies that $\lambda a\in H_i$ for some $i.$ We will prove that $\lambda a\in {\rm NP}(I_i).$ 
By the definition of ${\rm NP}(I)$ one has  $\lambda a=u+v,$ where $u$ is a convex linear combination of elements $b_j \in E$ and $v\in \bR^n_{\ge 0}.$ 
Since $\lambda a\in H_i$ and each $b_j$ is on or above $H_i,$ it follows that each $b_j$ is in $H_i$ and therefore $b_j \in E\cap H_i=E_i.$ Thus $\lambda a\in {\rm NP}(I_i)$ and so $a\in {\rm NP}(I_i)$ as asserted. 
\end{proof}

\begin{Lemma}\label{pointsinInterior} We use Notation  ~\ref{settingSection7}. If  $\alpha$ is a monomial 
in $I\setminus J,$ then the image $\alpha + \m I$ of $\alpha$ in $[\cF(I)]_1$  
is nilpotent in $\cF(I).$
\end{Lemma}
\begin{proof}
Since $\alpha \notin J$, we have $v_i(\alpha) > d_i $ for all $i.$  On the other hand,    $\overline{I}=\sum\overline{I_i}$ by Proposition \ref{Prop7.2}, hence, $\alpha \in \overline{I_i}$ for some $i$ because $\alpha$ and all the ideals $\overline{I_i}$ are monomial. 
Thus,  $\alpha^t \in I_i^t$ for some $t>0,$ again because $\alpha$ and $I_i$ are monomial.
From $v_i(\alpha)> d_i$ we obtain $v_i(\alpha^t) > td_i$ and hence $\alpha^t \in \m I_i^t\subset \m I^t$.
We conclude that $\alpha+\m I$ is nilpotent in $\cF(I),$ as claimed.
\end{proof}

\begin{Theorem}\label{Th7.1}
We use Notation  ~$\ref{settingSection7}$.
The inclusion of ideals $J\subset I$ induces an isomorphism
$$
\frac{\cF(J)}{\sqrt{0}} \cong  \frac{\cF(I)}{\sqrt{0}}. 
$$
\end{Theorem}

\begin{proof}
We consider the inclusion of Rees rings $\cR(J)\subset \cR(I),$ which is an integral extension by Proposition \ref{Prop7.2}. It induces a homomorphism of graded $k$-algebras
\[\varphi :  \frac{\cF(J)}{\sqrt{0}} \lto \frac{\cF(I)}{\sqrt{0}}.
\]
This map is surjective by Lemma \ref{pointsinInterior}. 

It remains to prove that $\varphi$ is injective, equivalently,  that  the nilradical of $\cF(I)$ contracts to the nilradical of $\cF(J)$ under the natural map $\cF(J) \lto \cF(I)$. This holds if every prime ideal of $\cF(J)$ is contracted from a prime ideal of $\cF(I),$ equivalently, if every prime ideal of $\cR(J)$ containing $\m$ is contracted from a prime ideal of $\cR(I)$ containing $\m.$ The latter holds because the integral extension $\cR(J)\subset \cR(I)$ satisfies lying over.
\end{proof}

\begin{Theorem}\label{Th7.2}
In the setting of Notation ~$\ref{settingSection7}$, 
there are homomorphisms of graded $k$-algebras 
$\psi_i: \cF(I) \lto \cF(I_i)$
with  
$$
\psi_i(\alpha_j+\m I)=  \begin{cases}
\alpha_j +\m I_i & \text{ if } \alpha_j \in \mathcal{G}_i\, , \\
0  & \text{ if } \alpha_j \in \mathcal{G} \setminus \mathcal{G}_i \, .
\end{cases}
$$
One has $\psi_i \circ \varphi_i={\rm id}, $ 
in particular, $\varphi_i$ are injective and $\psi_i$ are surjective. Furthermore,
$\, \ker \psi_i=L_i:= (\alpha_j +\m I \ | \ \alpha_j \in \mathcal{G} \setminus \mathcal{G}_i) \cF(I).$
\end{Theorem}

\begin{proof}
Consider the homomorphism of graded $R$-algebras $R[Y] \twoheadrightarrow \mathcal R(I)=R[It]$ with $y_j \mapsto \alpha_jt.$
Since $\cR(I)$ is a monomial algebra, the kernel of this map is a toric ideal, generated by binomials of the form $aM-bN$, where $M,N$ are monomials in $Y$ of the same degree, and $a,b$ are monomials in $x_1, \ldots, x_n$.
The kernel of the induced map $k[Y] \twoheadrightarrow \cF(I)$
is generated by the images of these binomials, and it suffices to consider  binomials with $a= 1$ or $b=1$.
We claim that each of these binomials remains a relation in $\cF(I_i)$.
This shows that the retraction $k[Y] \twoheadrightarrow k[Y_i]$
 defined by
$$
y_j \mapsto \begin{cases}
y_j & \text{ if } \alpha_j \in \mathcal{G}_i \,, \\
0  & \text{ if } \alpha_j \in \mathcal{G} \setminus \mathcal{G}_i \,,
\end{cases}
$$
induces a map on the level of special fiber rings, as desired.

If there exist  $j_1,j_2$ such that $y_{j_1}$ divides $M$, $y_{j_2}$ divides $N$,
and both $\alpha_{j_1},\alpha_{j_2}$ do not belong to $\mathcal{G}_i$, then the binomial is mapped to $0$ in $\cF(I_i).$
If $M$ and $N$ only involve variables $y_j$ such that $\alpha_j \in \mathcal{G}_i$, then the binomial is still a relation in $\cF(I_i)$.
Thus, we may assume that $M$ only involves variables $y_j$ such that $\alpha_j \in \mathcal{G}_i$,
while 
$N$ is divisible by a variable $y_h$ such that $\alpha_h \in \mathcal{G} \setminus \mathcal{G}_i$.
As $\alpha_h \in \mathcal{G}\setminus \mathcal{G}_i$, we have $v_i(\alpha_h)>d_i.$
Since $M,N$ have the same degree,
it follows that $v_i(M(\mathcal{G})) < v_i(N(\mathcal{G}))$.
On the other hand, $aM(\mathcal{G})= bN(\mathcal{G}) $ in $R$,
hence, $v_i(a)>0,$ which implies $a \in \m$.
Thus, the binomial $aM-bN$ is mapped to zero in $\cF(I_i)$.

Now that we have proved that the maps $\psi_i$ are well defined, it follows 
immediately that $\psi_i \circ \varphi_i={\rm id}, $ as claimed. 
Clearly $L_i \subset \ker \psi_i.$ 
Hence, there are induced maps
\[
\cF(I_i) \xrightarrow{\overline{\varphi_i}} 
\cF(I)/L_i \xrightarrow{\overline{\psi_i}} 
\cF(I_i)
\]
with $\overline{\psi_i} \circ \overline{\varphi_i} = {\rm id}$. Notice that $\overline{\varphi_i}$ is injective. But 
$\overline{\varphi_i}$ is also surjective by the definition of $L_i,$ so this map is bijective, showing that $\overline{\psi_i}$
is bijective as well. This proves the equality $\ker \psi_i = L_i.$
\end{proof}

\begin{Lemma}\label{intersectionNil} With notation as in Theorem $\ref{Th7.2}$, the ideal $\, \bigcap\limits_i L_i \subset \cF(I)$ is nilpotent. 
\end{Lemma}
\begin{proof} It suffices to prove that $\prod\limits_{j=1}^s L_j$ is nilpotent in $\cF(I).$ This ideal is generated by elements of the form 
\[(\alpha_1+\m I)\cdot \ldots \cdot (\alpha_s +\m I) = \prod\limits_{j=1}^{s}\alpha_j +\m I^s
\]
with $\alpha_j\in \mathcal{G} \setminus \mathcal{G}_j.$
Since $v_i(\alpha_i)>d_i,$ it follows that 
$\, v_i(\prod\limits_{j=1}^{s}\alpha_j)> s \, d_i=v_i(I^s)$
for every $i.$ 
So, applying Lemma~\ref{pointsinInterior} to the ideal $I^s$, we see that $\prod\limits_{j=1}^{s}\alpha_j +\m I^s$ is nilpotent in $\cF(I^s),$ hence in $\cF(I).$
\end{proof}

\begin{Notation}\label{kernels}{\rm Adopt the notation of Notation  ~\ref{settingSection7} and of Theorem \ref{Th7.2}. 
Let 
$$\pi: k[Y]\lto \mathcal F(I)/\sqrt{0}$$
be the homogeneous epimorphism with $\pi(y_j)=\alpha_j +\m I,$ and let
$$\pi_i: k[Y_i]\lto \mathcal F(I_i)/\sqrt{0}$$
be the homogeneous epimorphism with $\pi(y_j)=\alpha_j +\m I_i.$ Write $K=\ker \pi$ and $K_i=\ker \pi_i$ for the defining ideals of the reduced special fiber rings. 

Let $\p_{i\ell}$ be the minimal primes of $\cF(I_i)$ and let $\q_{i\ell}=\pi_i^{-1}(\p_{i\ell})$ be their preimages in $k[Y_i].$ Finally, write $P_{i\ell}=(\p_{i\ell}, L_i)\cF(I)$ and $Q_{i\ell}=\pi^{-1}(P_{i\ell})$ for their preimages in $k[Y].$
}
\end{Notation}

\begin{Theorem}\label{ThmReducedSpecialFiberIdealDescription} In the setting of  Notation $\ref{kernels}$. \begin{enumerate}[{$(a)$}]
    \item The minimal prime ideals of $\cF(I)$ are the minimal elements of the set $\{ P_{i\ell}\}.$
    \item $K=\bigcap\limits_i\ (K_i, Y\setminus Y_i) =\bigcap\limits_{i,\ell}\ (\q_{i\ell}, Y \setminus Y_i)=\bigcap\limits_{i,\ell} \, Q_{i\ell}.$
\end{enumerate}
    
\end{Theorem}
\begin{proof}
    To prove part (a), 
    it suffices to show that for every minimal prime ideal $ P$ of $\cF(I)$ we have $ P\in \{ P_{i\ell}\}.$ 
    It follows from Lemma~\ref{intersectionNil} that $\bigcap\limits_{i,\ell} L_i\subset P,$ thus, $L_i \subset P $ for some $i.$ 
    Therefore, $P/L_i$ is a minimal prime ideal of $\cF(I)/L_i$  and so, under the isomorphism of Theorem ~\ref{Th7.2}, the ideal $P/L_i$ is mapped onto a  minimal prime ideal $\p_{i\ell}$ of $\cF(I_i).$ This shows that $P=P_{i\ell}.$

    The equality $K=\bigcap\limits_{i,\ell} \, Q_{i\ell}$ of part (b) follows immediately from (a), and the equality  $\bigcap\limits_{i,\ell} \, Q_{i\ell}=\bigcap\limits_{i,\ell}\ (\q_{i\ell}, Y \setminus Y_i)$ is obvious. Finally, $K_i$ and $\q_{i\ell}$ are ideals of $k[Y_i]$ with $K_i=\bigcap\limits_{\ell} \, \q_{i\ell}.$ 
Therefore, $(K_i, Y\setminus Y_i)=\bigcap\limits_{\ell} \ (\q_{i\ell}, Y\setminus Y_i),$ which proves the second equality in part (b). 
\end{proof}

We are now going to focus on Rees valuations centered at the maximal ideal $\m,$ which was also the case considered in most of Section~\ref{SectionBehrend}. We denote these Rees valuations by $v_1, \ldots, v_t.$

\begin{Proposition}\label{Prop7.9} In the setting of  Notation $\ref{kernels}$, if the valuation $v_i$ is centered at the maximal ideal $\m,$ then the natural map $ k[\mathcal{G}_i]\twoheadrightarrow \cF(I_i)$ is an isomorphism. 
In particular, $\cF(I_i)$ is a domain, $\{\p_{i\ell} \}=\{0\},$ 
$\{ P_{i\ell}\}=\{L_i\},$ $\{ Q_{i\ell}\}=\{ (K_i, Y\setminus Y_i) k[Y]\}.$ Furthermore, $L_1, \ldots, L_t$
are distinct minimal primes of $\cF(I).$ \end{Proposition}
\begin{proof}
Assigning to each variable $x_j$ degree $v_i(x_j)>0,$ the ring $R$ is positively graded and $I_i$ is generated in a single degree. Therefore, Proposition~\ref{PropReducedFiberHomogeneousIdeal} implies that the map $k[\mathcal{G}_i] \twoheadrightarrow \cF(I_i)$ is an isomorphism. 
The fact that $L_1, \ldots, L_t$ are distinct minimal primes of $\cF(I)$ will be an immediate consequence
of Discussion~\ref{Discussion7.10}.
\end{proof}

\begin{Discussion}\label{Discussion7.10}{\rm  We saw in Proposition~\ref{1-1} that there is a one-to-one correspondence
between the Rees valuations $v_1, \ldots,v_s$ of $I$ and the minimal primes $\fP_1, \ldots, \fP_s$
of $I\cR(I),$ and so there is a one-to-one correspondence
between $v_1, \ldots, v_t$ and the minimal primes $\fP_1, \ldots, \fP_t$
of $I\cR(I)$ that contain $\m.$ 
For $1 \leq  i \leq t \, $, the $k$-algebra $A$ of Notation and Discussion~\ref{Dgrading} is equal to the ring $k[\mathcal{G}_i]$ of Proposition~\ref{Prop7.9}, thus, Lemma~\ref{LIso} gives a surjective homomorphism 
$\Psi_i: \cR(I) \twoheadrightarrow k[\mathcal{G}_i]$ with $\ker \Psi_i=\fP_i,$ and this map fits into the 
commutative diagram
\[
\begin{tikzcd}
\mathcal{R}(I) \arrow[->>]{r}[above]{\rm nat} \arrow[->>]{d}{\Psi_i} & \mathcal{F}(I) \arrow[->>]{d}{\psi_i} \\
k[\mathcal{G}_i] \arrow[->>]{r}[above]{\rm nat}[below]{\sim} & \mathcal{F}(I_i) \, ,
\end{tikzcd}
\]
where the isomorphism holds by Proposition~\ref{Prop7.9}.
As $\ker \psi_i=L_i$ we see that $\fP_i \subset \cR(I)$ is the preimage of $L_i \subset \cF(I).$

From this we deduce that, first, $\fP_i=(\m, L_i)\cR(I)$ and, second, that the ideals $L_1, \ldots, L_t$ are indeed distinct minimal primes of $\cF(I).$ In fact, these are
precisely the minimal primes of $\cF(I)$ that have maximal dimension, namely $n,$ 
as can be seen from the fact that the ring $\cR(I)/I \cR(I) \cong {\rm gr}_I(R)$
is equidimensional of dimension $n$ and maps onto $\cF(I).$
}\end{Discussion}

\begin{Corollary} The minimal primes $\fP_1, \ldots, \fP_t$ of $I \cR(I)$ that contain $\m$ 
are of the form $\fP_i=(\m, L_i)\cR(I)$ and their preimages in the polynomial ring $R[Y]$ 
are $(\m, K_i, Y\setminus Y_i)R[Y],$ with $L_i$ and $K_i$ as defined in Theorem~$\ref{Th7.2}$ and Notation~$\ref{kernels}$.
\end{Corollary}

It was described in Notation and Discussion~\ref{Dnozerodim} how to obtain all minimal prime ideals
of $I \cR(I)$ from the ones that contain $\m.$ This also gives a description of the
minimal prime ideals of the associated graded ring $G={\rm gr}_I(R).$

Theorem \ref{ThmReducedSpecialFiberIdealDescription} and Proposition \ref{Prop7.9} imply the next result for zero-dimensional monomial ideals, because
in that case all Rees valuations are centered at the maximal ideal.

\begin{Corollary}\label{Corollary6.12} In the setting of Notation $\ref{kernels}$, assume that $I$ is zero-dimensional. 
  \begin{enumerate}[{$(a)$}]
    \item The minimal primes of $\cF(I)$ are $L_1, \ldots, L_s.$
    \item $K=\bigcap\limits_{i=1}^s\ ( K_i, Y\setminus Y_i)$ is an irredundant intersection of
    prime ideals.
\end{enumerate}  
\end{Corollary}

\begin{Remark}\label{RemarkmaxAnalyicSpread} In the setting of Notation $\ref{kernels}$, 
assume that $v_1, \ldots, v_t$ are the Rees valuations centered at the maximal ideal $\m$. 
The following are equivalent:
 \begin{enumerate}[{$(1)$}]
    \item the minimal primes of $\cF(I)$ are $L_1, \ldots, L_t$;
    \item $K=\bigcap\limits_{i=1}^s\ ( K_i, Y\setminus Y_i)$ is an irredundant intersection of
    prime ideals;
    \item $\cF(I)$ is equidimensional of dimension $n.$
\end{enumerate} 
   \end{Remark}
\begin{proof} The equivalence of (1) and (2) follows the definition of $L_i,$ $K_i,$ $K$ 
in Theorem~\ref{Th7.2} and Notation~\ref{kernels}, and the equivalence of (2) and (3) is
a consequence of Discussion~\ref{Discussion7.10}.
    \end{proof}

In general, the minimal primes of the special fiber ring of a monomial 
ideal do not all arise from Rees valuations centered at the maximal
ideal, since these rings need not be equidimensional:
For instance, the ideal $(x_1^5x_2^5,x_1^3x_2^2x_3,_1x_2^2x_3^3,x_2^4x_3^3,x_1x_3^5)\subset k[x_1,x_2,x_3]$ has analytic spread 3, but its special fiber ring has a minimal prime ideal of dimension 2.

 Even when there exists  a non-negative grading of $R$ such that a monomial ideal is generated in a single degree, the special fiber ring is not necessarily a domain unless $R$, with the chosen grading, is positively graded. This is the case for the ideals $I_i$ associated to Rees valuations $v_i$ that are not centered at the maximal ideal. 
 In particular, if $\cF(I_i)$ is not a domain, it cannot be isomorphic to the toric ring $k[\mathcal{G}_i]$. We give some examples of this behavior below;
 they also show that,
 although each Rees valuation centered at the maximal ideal corresponds to a unique
 minimal prime of the special fiber ring, other Rees valuations may correspond to
 no or to multiple minimal primes.

\begin{Example}\label{Ex7.15}{\rm Let $I=(x^2,\, xy,\, y^2z)\subset k[x,y,z]$. Setting $\deg(x)=1, \deg(y)=1, \deg(z)=0$, this ideal is generated in degree 2. 
However, the special fiber ring is not isomorphic to $k[x^2, xy, y^2z]$. 
Indeed, the defining ideal of the special fiber ring is $(y_1y_3)\subset k[y_1, y_2, y_3],$ in particular, 
$\cF(I)$ has two minimal primes. 
The ideal $I$ has two Rees valuations, one gives the ideal $I_1=(xy,\, y^2z)$ and the other gives $I_2=I.$ 
Since $I_1 \subsetneq I_2$, 
the ideal $I_1$ cannot correspond to a
minimal prime of $\cF(I),$ whereas $I_2$ gives rise to two minimal primes.

In general, the special fiber ring of an equigenerated monomial ideal may not be reduced. 
For instance, the special fiber ring of $I=(x^3,\, xy^2,\, y^3z)\subset k[x,y,z]$ 
has defining ideal 
$(y_1y_3^2)\subset k[y_1,y_2,y_3].$
Note that, setting  $\deg(x)=1, \deg(y)=1, \deg(z)=0$ the ideal $I$ is generated in degree 3. 
}\end{Example}

\begin{Remark}{\rm
In \cite{Singla}, Singla computes the reduced special fiber ring of monomial ideals that are \emph{extremal}, that is,  equal to their minimal monomial reduction. 
In that case, she obtains a formula similar in appearance to our Theorem~\ref{ThmReducedSpecialFiberIdealDescription}(b). 
However, the ideals $I_i$ she considers do not correspond to Rees valuations, but to the maximal compact faces of the Newton polyhedron. 
For instance, in the first example of Example~\ref{Ex7.15}, our ideals are $I_1 = (xy,\, y^2z)$ and $I_2 = I$, whereas hers are $I_1 = (x^2,\, xy)$ and $I_2 = (xy,\, y^2z)$. 
Notice that the convex hull of the exponent vectors of the monomials in our $\mathcal{G}_2$ is not a face of ${\rm NP}(I),$  because it properly contains the maximal compact faces given by the convex hulls of the exponent vectors of her $\mathcal{G}_1$ and $\mathcal{G}_2$. 
In general, as we explained in Example~\ref{Ex7.15}, $\cF(I_i)$ may not coincide with the toric ring $k[\mathcal{G}_i]$. 
In her case however, since $I$ is extremal, the exponent vectors of the generators her ideals $I_i$ are the vertices of maximal compact faces of ${\rm NP}(I).$ 
Therefore, the ideals $I_i$ have no proper reduction, so their special fiber rings $\cF(I_i)$ are polynomial rings and thus agree with the toric rings $k[\mathcal{G}_i]$. 
}\end{Remark}

\begin{Remark}{\rm In the zero-dimensional case, the Rees valuations 
do correspond to the maximal compact faces of ${\rm NP}(I).$ 
Thus,
in this case our ideals $I_i$ coincide with those of Singla, so our formula is a 
generalization of hers, from extremal monomial ideals to arbitrary monomial ideals.
This raises the question of whether Singla's formula can be extended to 
ideals that are not necessarily zero-dimensional.}
\end{Remark}
 
\subsection*{Acknowledgments}

We are grateful to Michele Graffeo and Andrea Ricolfi for their inspiring article and for many helpful discussions, and to Winfried Bruns, Alessio D'Alì, Mircea Musta\c{t}\u{a} and Ritvik Ramkumar for helpful conversations on the material of this paper.
Computations with Macaulay2 \cite{GS} assisted us  during the preparation of this paper.

This material is based in part upon work supported by the National Science Foundation under Grant No.~DMS-1928930, while two of the authors were in residence at the Simons Laufer Mathematical Sciences Institute (SLMath) in Berkeley, California, during the Spring 2024 semester on Commutative Algebra, and while the third author was visiting SLMath. We thank SLMath for its hospitality.

\bibliographystyle{amsra}
\bibliography{references}

\end{document}